\documentclass[12pt,twoside,reqno]{amsart}

\usepackage{mathpazo,amsfonts,nicefrac}
\usepackage[letterpaper, portrait, margin=1in]{geometry}
\usepackage{amsmath}
\usepackage{amssymb}
\usepackage{amsfonts}
\usepackage{fancyhdr}
\usepackage{MnSymbol}
\usepackage{mathrsfs}
\usepackage{indentfirst}
\usepackage{mathtools} 
\usepackage{commath}
\usepackage{soul}
\usepackage{upgreek} 
\usepackage{graphicx} 
\usepackage[rightcaption]{sidecap}
\usepackage{wrapfig}
\usepackage{caption}
\usepackage{theoremref}
\usepackage{cancel}
\usepackage{amsthm}
\usepackage{xcolor}
\usepackage[final]{hyperref}
\usepackage{float}
\usepackage{comment}

\hypersetup{unicode= false, colorlinks=true, linkcolor=blue,
	anchorcolor=blue, citecolor=green, filecolor=red, menucolor=blue, urlcolor=blue}

\newtheorem{theorem}{Theorem}[section]

\newtheorem{lemma}[theorem]{Lemma}
\newtheorem{definition}[theorem]{Definition}
\newtheorem{example}[theorem]{Example}
\newtheorem{corollary}[theorem]{Corollary}
\newtheorem{remark}[theorem]{Remark}

\def\C{{\mathbb{C}}}
\def\R{{\mathbb{R}}}

\def\BB{{\mathcal{B}}}

\title{Periodic approximations in inverse spectral problems \\ for canonical Hamiltonian systems}

\begin{document}

\setlength{\footskip}{14.0pt}.

\thispagestyle{empty}

\author{A.~Poltoratski}
	\address{University of Wisconsin\\ Department of Mathematics\\ Van Vleck Hall\\
		480 Lincoln Drive\\
		Madison, WI  53706\\ USA }
	\email{poltoratski@wisc.edu}
	\thanks{The first author is supported by
		NSF Grant DMS-1954085.}
		
\author{A. R.~Zhang}
	\address{University of Wisconsin\\ Department of Mathematics\\ Van Vleck Hall\\
		480 Lincoln Drive\\
		Madison, WI  53706\\ USA }
	\email{ashleyrzhang@wisc.edu}

\begin{abstract}
    This note is devoted to inverse spectral problems for canonical Hamiltonian systems on
    the half-line. An approach to inverse spectral problems based on the use of truncated Toeplitz operators 
    has been especially effective in the case when the spectral measure of the system is a locally finite periodic measure (see \cite{MP}). In this note we extend the periodic algorithm to the case of non-periodic measures by considering periodizations of a spectral measure and
    showing that the Hamiltonians corresponding to the periodizations converge to
    the Hamiltonian of the original measure.
\end{abstract}
	
	\maketitle

\section{Introduction}

 This paper studies spectral problems for canonical Hamiltonian systems on a half-line. A standard tool in the study of such problems is the Krein-de Branges (KdB) theory. The theory was developed in the middle of the 20th century by M. G. Krein, who noticed multiple connections between structural problems in certain spaces of entire functions and spectral problems for canonical systems. L. de Branges later developed the complex analytic part of the theory. The basics of KdB-theory and further references can be found in the original book by de Branges \cite{dB}, a chapter in the book by Dym and McKean \cite{DM}, as well as in recent monographs by C. Remling \cite{Remling} and R. Romanov \cite{Rom}. In the last 20 years, the theory experiences a new peak of activity because of its connection to other parts of analysis, such as the theory of orthogonal polynomials and the non-linear Fourier transform (see for instance \cite{Denisov, Scatter}, and further connections to other areas of mathematics such as number theory,  random matrices and zeros of the zeta function (\cite{Burnol, Lagarias1, Lagarias2, Benedek}).

Canonical systems are $2 \times 2$ differential equation systems of the form \begin{equation} \label{CS}
    \Omega \Dot{X}(t) = z H(t) X(t), \ \ t \in (t_-, t_+),\ \ -\infty < t_- < t_+ \leqslant \infty,
\end{equation}
where $z\in\C$ is a spectral parameter,
\begin{equation*}
    \Omega = \begin{pmatrix} 0 & 1 \\ -1 & 0 \end{pmatrix}
    \end{equation*}
    is the symplectic matrix,
\begin{equation*}
   H(t) = \begin{pmatrix} h_{11}(t) & h_{12}(t) \\ h_{21}(t) & h_{22}(t) \end{pmatrix}
\end{equation*}
is a given matrix-valued function called the Hamiltonian of the system and
\begin{equation*}
    X(t) = \begin{pmatrix} u(t) \\ v(t) \end{pmatrix}, 
    \end{equation*}
    is the unknown vector-function.

We make the following assumptions on the Hamiltonian $H(t)$:

\begin{itemize}
    \item $H(t) \in \mathbb{R}^{2 \times 2}$,
    \item $H(t) \in L^1_{loc}(\mathbb{R})$,
    \item $H(t) \geqslant 0$ almost everywhere,
    \item $\det(H) \neq 0$ almost everywhere.
\end{itemize}

Our fourth condition implies in particular that the system
does not have any 'jump intervals' (subintervals of $(t_-,t_+)$ on which $H$ is a constant matrix of rank 1).

A new approach to inverse spectral problems for canonical systems was recently used in 
\cite{BR, Bessonov, MP}. The approach is based on the study of truncated Toeplitz operators 
with symbols equal to the spectral measures of canonical systems. It allows one to extend the set of systems considered in classical texts by Borg, Marchenko and Gelfand-Levitan, and to find new explicit examples of solutions for inverse spectral problems. The Toeplitz approach applies to systems
whose spectral measures belong to the PW-class, the class
of sampling measures in Paley-Wiener spaces, which is significantly
broader than the class of spectral measures of Schr\"odinger
operators and Dirac systems with summable potentials considered in classical theory.

The new approach produces an especially clear finite-dimensional algorithm in the case when the spectral measure is a locally finite periodic measure, see \cite{MP}. Our goal in this note is to extend
this case of the algorithm beyond the class of periodic measures. For a general spectral measure
$\mu$ one can consider its 'periodization' $\mu_T$ (see Section \ref{Periodization}) and apply the tools for the periodic case to find the corresponding Hamiltonian $H_T$.
Our main results, theorems \ref{PW}, \ref{decay} and \ref{polygrowth}, say that the Hamiltonians $H_T$ converge to the Hamiltonian $H$ corresponding to $\mu$, as $T\to\infty$, for certain classes of measures. Theorem \ref{PW} shows convergence in the PW-class
while  theorems \ref{decay} and \ref{polygrowth} show
that in some instances the periodization approach  works beyond the PW-class. To our knowledge, apart
from general existence and uniqueness results,
spectral problems outside of PW-class have not have not been treated in the literature previously.

The convergence of $H_T$ to $H$ established in our main results is weak convergence on test
functions, see section \ref{Periodization}. Numerical evidence suggests that for a general
locally summable Hamiltonian these statements
cannot be strengthened to convergence in $L^1_{loc}$ or other kinds of stronger convergence
without additional assumptions on the system. Establishing exact metrics of convergence in various classes of systems seems to be an interesting open question.

The paper is organized as follows: In section \ref{KdBTheory}, we review the basics of Krein-de Branges theory of canonical systems. In section \ref{MPAlgorithm}, we revisit the inverse spectral problem algorithm for periodic Paley-Wiener measures \cite{MP}. In section \ref{Periodization}, we present our main results. Throughout the paper we supplement our statements with examples of solutions to the inverse spectral problem. In particular, section \ref{Periodization} contains
new examples of convergence of periodizations illustrating theorems \ref{PW}, \ref{decay} and \ref{polygrowth}.

\section{Canonical systems and de Branges spaces}\label{KdBTheory}

Via the change of variable $ds=\det H(t) dt$, we can normalize the system to satisfy $\det(H(s)) = 1$ almost everywhere. We call such systems det-normalized and assume this normalization throughout the rest of the paper.

A solution to \eqref{CS} is a $C^2-$function 
$$X(t) = X_z(t) = \begin{pmatrix} u_z(t) \\v_z(t) \end{pmatrix}$$
on $(t_-,t_+)$ satisfying the equation. An initial value problem (IVP) for \eqref{CS} is given by an initial condition $X(t_-) = c$, $c \in \mathbb{R}^2$. 
Every IVP for \eqref{CS} has a unique solution $X_z(t)$ on $(t_-, t_+)$, see for instance \cite{Remling}.

We use $H^2(\mathbb{C}_+)$ to denote the Hardy space on the upper half-plane $\mathbb{C}_+$: \begin{equation*}
    H^2(\mathbb{C}) = \{ f \in \mathcal{H}(\mathbb{C_+})~|~ \sup_{y > 0} \int_{\mathbb{R}} \abs{f(x + iy)}^2 < \infty\},
\end{equation*}
where $\mathcal{H}(C_+)$ is the set of all analytic functions in $\mathbb{C_+}$.

An entire function $E(z)$ is called an \textit{Hermite-Biehler} function if $|E(z)| > |E(\overline{z})|$ for $z \in \mathbb{C}_+$.
Given an entire function $E(z)$ of Hermite-Biehler class, define the \textit{de Branges' space} $B(E)$ based on $E$ as \begin{equation*}
    B(E) = \{F \text{~entire}\ |\ F/E, F^\#/E \in H^2(\mathbb{C_+})\},
\end{equation*}
where $F^\#(z) = \overline{F(\overline{z})}$.  The space $B(E)$ becomes a Hilbert space when endowed with the scalar product inherited from $H^2(\C_+)$: \begin{equation*}
    [F, G] = \int_{-\infty}^\infty F(t) \overline{G(t)} \frac{dt}{|E(t)^2|}.
\end{equation*}We call an entire function real if it is real-valued on $\R$. Associated with $E$ are two real entire functions  $A = \frac{1}{2}(E + E^\#)$ and $C = \frac{i}{2}(E - E^\#)$ such that $E=A-iC$.
De Branges' spaces are reproducing kernel Hilbert spaces. The reproducing kernels 
\begin{equation*}
    K_\lambda(z) = \frac{\overline{E(\lambda)} E(z) - \overline{E^\#(\lambda)} E^\#(z)}{2\pi i(\overline{\lambda} - z)} = \frac{\overline{C(\lambda)} A(z) - \overline{A(\lambda)}C(z)}{\pi(\overline{\lambda} - z)}
\end{equation*}
are functions from $B(E)$ such that  and $[F, K_{\lambda}] = F(\lambda)$ for all $F \in B(E)$, $\lambda \in \mathbb{C}$.

De Branges' spaces have an alternative axiomatic definition, which is useful in many applications.

\begin{theorem}[\cite{dB}] \label{dBAxioms}
Suppose that $H$ is a Hilbert space of entire functions that satisfies \begin{itemize}
\item $F \in H$, $F(\lambda) = 0$ $\Rightarrow$ $F(z) \frac{z - \overline{\lambda}}{z - \lambda} \in H$ with the same norm,
\item $\forall \lambda \notin \mathbb{R}$, the point evaluation is a bounded linear functional on $H$,
\item $F \to F^\#$ is an isometry.
\end{itemize}
Then $H = B(E)$ for some entire function $E$ of Hermite-Biehler class.
\end{theorem}

A de Branges space is called short (or regular) if
for any $F\in \BB(E)$ and any  $w\in\C$, $(F(z)-F(w))/(z-w)\in \BB(E)$. Such spaces are related to canonical systems with
locally summable Hamiltonians, see below.

Let $X_z(t) = \begin{pmatrix} u_z(t) \\v_z(t) \end{pmatrix}$  be the unique solution for \eqref{CS}  on $(t_-, t_+)$
 satisfying a self-adjoint initial condition $X(t_-) = x$, $x \in \mathbb{R}^2$. Then for each fixed $t$, the function $E_t(z) = u_z(t) - i v_z(t)$ is an Hermite-Biehler entire function. 

Every det-normalized canonical system delivers a chain of nested de Branges' spaces $B(E_t)$: if  $t_-\leqslant t_1 < t_2 \leqslant t_+$, then $B(E_{t_1}) \subseteq B(E_{t_2})$.
In absence of jump intervals, which follows from the $\det(H) \neq 0$ almost everywhere condition we imposed on $H(t)$, all such inclusions are isometric, see \cite{dB, Remling}.

A positive measure $\mu$ on $\R$  is called the spectral measure of \eqref{CS} corresponding to a fixed self-adjoint boundary condition at $t_-$ if all de Branges' spaces $B(E_t)$, $t \in [t_-, t_+)$ are isometrically embedded into $L^2(\mu)$. Under the restrictions we imposed on the system, such a spectral measure always exists and is unique if and only if \begin{equation*}
    \int_{t_-}^{t_+} \text{trace~} H(t) dt = \infty.
\end{equation*}
The case when the integral above is infinite (and the spectral measure is unique) is called the limit point case, otherwise it is called the limit circle case. 

We say that a positive measure $\mu$ on $\R$ is Poisson-finite if
$$\int_\R\frac{d\mu(x)}{1+x^2}<\infty.$$

Under the condition  $H(t) \in L^1_{loc}$, any spectral measure $\mu$ is Poisson-finite \cite{Remling}.

\section{Inverse spectral problems for periodic \texorpdfstring{$PW-$}{PW}measures}\label{MPAlgorithm}
\subsection{\texorpdfstring{$PW-$}{PW}measures and spaces}

We denote by $PW_t$ the standard Paley-Wiener space, the 
subspace of $L^2(\R)$ consisting of entire functions of exponential type at most $t$. Note that every $PW_t$ 
is a de Branges space, $PW_t=B(E),\ E(z)=e^{-itz}$.
The reproducing kernels of $PW_t$ are the sinc functions,
$$\frac{\sin(t(z-\lambda))}{\pi( z-\lambda)}.$$

\begin{definition}
A positive Poisson-finite measure $\mu$ is sampling for the Paley-Wiener space $PW_t$ if there exist constants $0 < c < C$ such that for all $f \in PW_t$ \begin{equation*}
    c \|f\|_{PW_t} \leqslant \|f\|_{L^2(\mu)} \leqslant C \|f\|_{PW_t}.
\end{equation*}

A positive  measure $\mu$ is a Paley-Wiener measure ($\mu \in PW$) if it is sampling for $PW_t$ for all $t > 0$.
\end{definition}

It is not difficult to show that any $PW$-measure is Poisson-finite. For this and other basic properties of $PW$-measures see \cite{MP}.

For a $ PW$-measure $\mu$ the $L^2(\mu)$-norm is equivalent to the norm in $PW_t$  for all $t>0$. We will denote by $PW_t(\mu)$ the Hilbert space
of $PW_t$-functions equipped with the $L^2(\mu)$ scalar product. Note that each $PW_t(\mu)$ is a de Branges' space because it satisfies the three axioms in theorem \ref{dBAxioms}.

Meanwhile, any Poisson-finite measure, and in particular any $PW$-measure,  is a spectral measure for a canonical system with locally integrable Hamiltonian $H$, see for instance \cite{Remling}. Recall that a fixed canonical system generates a chain of nested de Branges spaces $B(E_t)$ all isometrically embedded in $L^2(\mu)$. Under the restriction $H(t)\in L^1_{loc}$ all spaces are short and any spectral measure is Poisson-finite. Any Poisson-finite measure admits a unique chain of short de Branges' spaces isometrically embedded in $L^2(\mu)$, see \cite{Remling}. It follows that the chains $B(E_t)$ and $PW_t(\mu)$ defined above are the same. As shown in \cite{MP}, any canonical system whose spectral measure belongs to the $PW$-class  satisfies $\det H\neq 0$ a.e. and may therefore be det-normalized.  

We call  a canonical system a $PW$-system if the corresponding  de Branges' spaces $B(E_t) $  are  equal to  $PW_t$ as sets for all $t$.
In this case any spectral measure $\mu$ of the system is a PW-measure. It follows that $PW$-systems can be equivalently defined as those systems whose spectral measures belong to the PW-class.

The class of PW-measures admits the following elementary description. 

\begin{definition}
Let $\mu$ be a positive measure on $\R$. We call an interval $I \subseteq \mathbb{R}$ a $(\mu, \delta)$-interval if \begin{equation*}
    \mu(I) > \delta, \quad \text{and} \quad \abs{I} > \delta,
\end{equation*}
where $\abs{I}$ stands for the length of the interval $I$.
\end{definition}

\begin{theorem}\label{PWcondition}\cite{MP}
A positive Poission-finite measure $\mu$ is a Paley-Wiener measure if and only if \begin{itemize}
    \item $\sup_{x \in \mathbb{R}} (x, x + 1) < \infty$,
    \item For any $d > 0$, there exists $\delta > 0$ such that for all sufficiently large intervals $I$, there exist at least $d \abs{I}$ disjoint $(\mu, \delta)$-intervals intersecting $I$.
\end{itemize}

\end{theorem}

We say that a measure $\mu$ has locally infinite support if $\text{supp\ }\mu\cap [-C,C]$ is an infinite
set for some $C>0$, or equivalently if $\text{supp\ }\mu$ has a finite accumulation point.
For a periodic measure one can easily deduce the following

\begin{corollary}\label{PWper}
A positive locally finite periodic measure is a Paley-Wiener measure if and only if it has locally infinite support.
\end{corollary}

\subsection{Truncated Toeplitz operators}
Let $\phi \in L^\infty(\mathbb{R})$. The truncated Toeplitz operator $L_\phi: PW_a \to PW_a$ with the symbol $\phi$ is given by $f \mapsto L_\phi(f) = P(\phi f)$, where $P: L^2(\mathbb{R}) \to PW_a$ is the orthogonal projection.
The definition can be extended to more general symbols via quadratic forms. 
Let $\mu$ be a measure on $\mathbb{R}$. The truncated Toeplitz operator $L_\mu: PW_a \to PW_a$ is given by the relation $$\int_{\mathbb{R}}(L_\mu f) \overline{g} dx = \int_{\mathbb{R}} f \overline{g} d\mu,$$ 
where $f, g \in PW_a$. 

\begin{lemma}[\cite{MP}]
$L_\mu$ is a positive invertible operator in $PW_a$ if and only if $\mu$ is a sampling measure for $PW_a$. Consequently, $L_{\mu}$ is a positive invertible operator in every $PW_a$, $0 < a < \infty$, if and only if $\mu$ is a PW-measure.
\end{lemma}

Consider a canonical system $\Omega \Dot{X}(t) = z H(t) X(t)$ with a fixed boundary condition at $t = t_-$ and a diagonal locally integrable Hamiltonian \begin{equation*}
    H(t) = \begin{pmatrix} h_{11}(t) & 0 \\ 0 & h_{22}(t)\end{pmatrix}.
\end{equation*}

Such a system corresponds to an even Poisson-finite spectral measure $\mu$ on the real line. Let $B_t = B(E_t)$ be the corresponding chain of de Branges' spaces, with reproducing kernels $K_\lambda^t(z) \in B_t$. 

The following approach to the inverse spectral problem was
presented in 
\cite{MP}. First, one can show that $H(t)$ can be recovered from the reproducing kernels by \begin{equation}\label{h11kernel}
    h_{11}(t) = \pi \frac{d}{dt} \| K_0^t \|^2 = \pi \frac{d}{dt} K_0^t(0), \quad  \text{and~} h_{22}(t) = \frac{1}{h_{11}(t)}.
\end{equation}

The reproducing kernels can be obtained from the sinc functions via truncated Toeplitz operators:
\begin{equation*}
    K_0^t(z) = L_\mu^{-1} \left( \frac{\sin(tz)}{\pi z} \right).
\end{equation*}

We will use the following standard notations for the unitary Fourier transform throughout the paper: \begin{equation*}
\begin{split}
    \mathcal{F}(f)(\xi) &= \hat{f}(\xi) = \frac{1}{\sqrt{2\pi}} \int_{-\infty}^\infty f(x) e^{-ix\xi} dx, \\
    \mathcal{F}^{-1}(f)(x) &= \check{f}(x) = \frac{1}{\sqrt{2\pi}} \int_{-\infty}^\infty f(\xi) e^{ix\xi} d\xi.
\end{split}    
\end{equation*}

Let $\mu$ be a spectral measure of a PW-system \begin{equation*}
    \Omega \Dot{X}(t) = z H(t) X(t).
\end{equation*}

Let $K_0^t$ be the reproducing kernels in the corresponding de Branges' spaces $B_t$. Denote by $f_t$ the Fourier transform of $K_0^t(x)$. Then $f_t(x)$ is supported on $[-t, t]$ and satisfies \begin{equation}\label{conveq}
    f_t \ast \hat{\mu}(x) = 1 \text{~on~} [-t, t].
\end{equation}
We can recover $H(t)$ from $f_t$ by \begin{equation}\label{h11f}
\begin{split}
    h_{11}(t) &= \pi \frac{d}{dt} K_0^t(0) = \pi \frac{d}{dt} \mathcal{F}^{-1}(f_t)(0) = \sqrt{\frac{\pi}{2}} \frac{d}{dt} \int_{-t}^t f_t(s)ds,\\
    h_{22}(t) &= \frac{1}{h_{11}(t)}.
\end{split}
\end{equation}

The following example illustrates the above formulas.

\begin{example}[\cite{MP}]\label{oneplusdelta}
Let $\mu = \sqrt{2\pi} \delta_0 + \frac{1}{\sqrt{2\pi}} m$, where $m$ is the Lebesgue measure on $\mathbb{R}$, and $\delta_0$ is the unit point mass at $0$. Then this is an even PW-measure, so we can recover $H(t)$ using the algorithm above.

Since $\hat{\mu} = m + \delta_0$, the Fourier transform of the reproducing kernel, $f_t$, satisfies \begin{equation*}
    1 = f_t \ast \hat{\mu}(x) = \int_{-t}^t f_t(s) ds + f_t(x), \quad \quad x \in [-t, t].
\end{equation*}

Since $\int_{-t}^t f_t(s) ds$ is a constant with respect to $x$, we have $f_t(x) = c(t) \mathbb{1}_{[-t, t]}(x)$, and
\begin{equation*}
    1 = \int_{-t}^t c(t) ds + c(t) = (2t + 1) c(t) \quad \quad \Rightarrow \quad \quad c(t) = \frac{1}{2t + 1}.
\end{equation*}

Recall that $h_{11}(t) = \sqrt{\frac{\pi}{2}} \frac{d}{dt} \int_{-t}^t f_t(s) ds$. Then \begin{equation*}
\begin{split}
    h_{11}(t) &= \sqrt{\frac{\pi}{2}} \frac{d}{dt} \int_{-t}^t \frac{1}{2t + 1} ds = \sqrt{\frac{\pi}{2}} \frac{d}{dt} \frac{2t}{2t + 1} = \sqrt{2\pi} \frac{1}{(2t + 1)^2},\\
    h_{22}(t) &= \frac{1}{h_{11}(t)} = \frac{(2t + 1)^2}{\sqrt{2\pi}}.
\end{split}
\end{equation*}
\end{example}

\subsection{Even periodic measures}\label{PeriodicPW}

The algorithm presented above takes an especially straightforward form in the case of a periodic spectral measure. If, in addition, one assumes that the measure is even, then the Hamiltonian is diagonal and one only needs to find the diagonal terms. The diagonal case corresponds to Krein's string operators, see for instance \cite{DM}.

If $a_n$ and $b_n$ are trigonometric moments of a given locally finite $2\pi$-periodic measure $\mu$,
$$a_0=\frac {\mu([-\pi,\pi])}{2\pi},\ \ a_n=\frac 1{\pi}\int_{-\pi}^{\pi}\cos (nx) \ d\mu(x),\ \ b_n=\frac 1{\pi}\int_{-\pi}^{\pi}\sin (nx) \ d\mu(x),\ n=1,2,...$$
we will formally write $\mu=\sum_{n = 0}^\infty (a_n cos(nx) + b_n \sin (nx))$.

\begin{theorem}[\cite{MP}] \label{2piThm}
Let $\mu$ be a spectral measure of a system \eqref{CS}. Suppose that $\mu$ is an even $2\pi$-periodic PW-measure, $\mu(x) = \sum_{n = 0}^\infty a_n cos(nx)$.

Consider the infinite Toeplitz matrix \begin{equation*}
    J = \begin{pmatrix} a_0 & \frac{a_1}{2} & \frac{a_2}{2} & \frac{a_3}{2}  & \ldots\\
    \frac{a_1}{2} & a_0 & \frac{a_1}{2} & \frac{a_2}{2}  & \ldots\\
    \frac{a_2}{2} & \frac{a_1}{2} & a_0 & \frac{a_1}{2}  & \ldots\\
    \frac{a_3}{2} & \frac{a_2}{2} & \frac{a_1}{2} & a_0  & \ldots\\
    \vdots & \vdots & \vdots & \vdots   & \ddots\\
    \end{pmatrix}.
\end{equation*}

Denote by $J_n$ the $(n + 1) \times (n + 1)$ sub-matrix on the upper-left corner of $J$. Then $h_{11}(t)$ in the Hamiltonian is a step function with \begin{equation*}
    h_{11}(t) = \Sigma(J_n^{-1}) - \Sigma(J_{n - 1}^{-1}) \; \text{on} \; \Big(\frac{n}{2}, \frac{n + 1}{2}\Big],\ n = 0, 1, 2, \ldots
\end{equation*}
where $\Sigma$ denotes the sum of all elements in the matrix, and $J_{- 1}^{-1} = 0$.
\end{theorem}

For reader's convenience we provide the proof. 

\begin{proof}
We obtain the formula for $h_{11}(t)$ in two steps: we first solve for $f_t(s)$ using \eqref{conveq}, then we use \eqref{h11f} to solve for $h_{11}(t)$.

Recall that $f_t$ is supported on $[-t, t]$ and satisfies \eqref{conveq}. Then $f_t$ satisfies \begin{equation*}
    f_t \ast \left(\sqrt{2\pi}a_0 \delta_0 + \sum_{k = 1}^n \sqrt{\frac{\pi}{2}} a_k (\delta_k + \delta_{-k}) \right) (x) = 1, \quad \quad x \in [-t, t],
\end{equation*}
where $n$ is the smallest integer such that $t \leqslant \frac{n + 1}{2}$.

Consider intervals $I_k$ centered at $-\frac{n}{2}, -\frac{n-2}{2}, \cdots,$  $\frac{n-2}{2}, \frac{n}{2}$, and of length $2(t - \frac{n}{2})$, enumerated from left to right. The complement of all the $I_k$'s on $[-t, t]$ consists of $n$ intervals, and we denote those  intervals by $\Tilde{I}_k$, again enumerated from left to right. The convolution equation on $I_k$'s and $\Tilde{I}_k$'s gives rise to $2n + 1$ linear equations.  Using all the equations, we get that $f_t(s)$ takes value $\alpha_k$ on $I_k$, and $\beta_k$ on $\Tilde{I}_k$, where 
\begin{equation*}
\begin{split}
    \begin{pmatrix} \alpha_1 & \alpha_2 & \ldots & \alpha_{n + 1} \end{pmatrix}^T = \frac{1}{\sqrt{2\pi}} J_n^{-1} \times 1,\\
    \begin{pmatrix} \beta_1 & \beta_2 & \ldots & \beta_{n} \end{pmatrix}^T = \frac{1}{\sqrt{2\pi}} J_{n - 1}^{-1} \times 1,
\end{split}
\end{equation*}
and $1$ is a column vector with all coordinates equal to $1$. 

Now we find $\int_{-t}^t f_t(s) ds$: \begin{equation*}
    \int_{-t}^t f_t(s) ds = \sum_{k = 1}^{n + 1} \alpha_k \cdot 2(t - n/2) + \sum_{k = 1}^n \beta_k (n - 2t + 1) = \frac{1}{\sqrt{2\pi}} \left( \Sigma J_n^{-1} (2t - n) + \Sigma J_{n - 1}^{-1} (n - 2t + 1) \right).
\end{equation*}

Therefore, \begin{equation*}
\begin{split}
    h_{11}(t) &= \sqrt{\frac{\pi}{2}} \frac{d}{dt} \int_{-t}^t f_t(s)ds = \sqrt{\frac{\pi}{2}} \frac{d}{dt}\frac{1}{\sqrt{2\pi}} \left( \Sigma J_n^{-1} (2t - n) + \Sigma J_{n - 1}^{-1} (n - 2t + 1) \right)\\
    &= \frac{1}{2} \left(2 \cdot \Sigma J_n^{-1} - 2 \cdot \Sigma J_{n - 1}^{-1} \right) = \Sigma(J_n^{-1}) - \Sigma(J_{n - 1}^{-1}).
\end{split}
\end{equation*}
\end{proof}

\begin{example}[\cite{MP}]
Consider the PW-system with $d\mu(x) = (1 + \cos x )dx$. 
Since this is a $2\pi$-periodic measure with infinite support on $[-\pi, \pi]$, it is a PW-measure. Therefore, we can apply the algorithm in the theorem to recover the corresponding Hamiltonian. The infinite Toeplitz matrix in this case is \begin{equation*}
    J = \begin{pmatrix} 1 & \frac{1}{2} & 0 & 0 & 0 & \ldots \\
    \frac{1}{2} & 1 & \frac{1}{2} & 0 & 0 & \ldots \\
    0 &  \frac{1}{2} & 1 & \frac{1}{2} & 0 & \ldots \\
    0 & 0 &  \frac{1}{2} & 1 & \frac{1}{2} & \ldots \\
    \vdots & \vdots & \vdots & \vdots & \vdots & \ddots
    \end{pmatrix}.
\end{equation*}

We here present snapshots of $f_t$ for $t = 2.5, 2.6, 2.75, 2.85, 3$. 

\begin{figure}[htbp]
\centering
\includegraphics[width=0.3\textwidth]{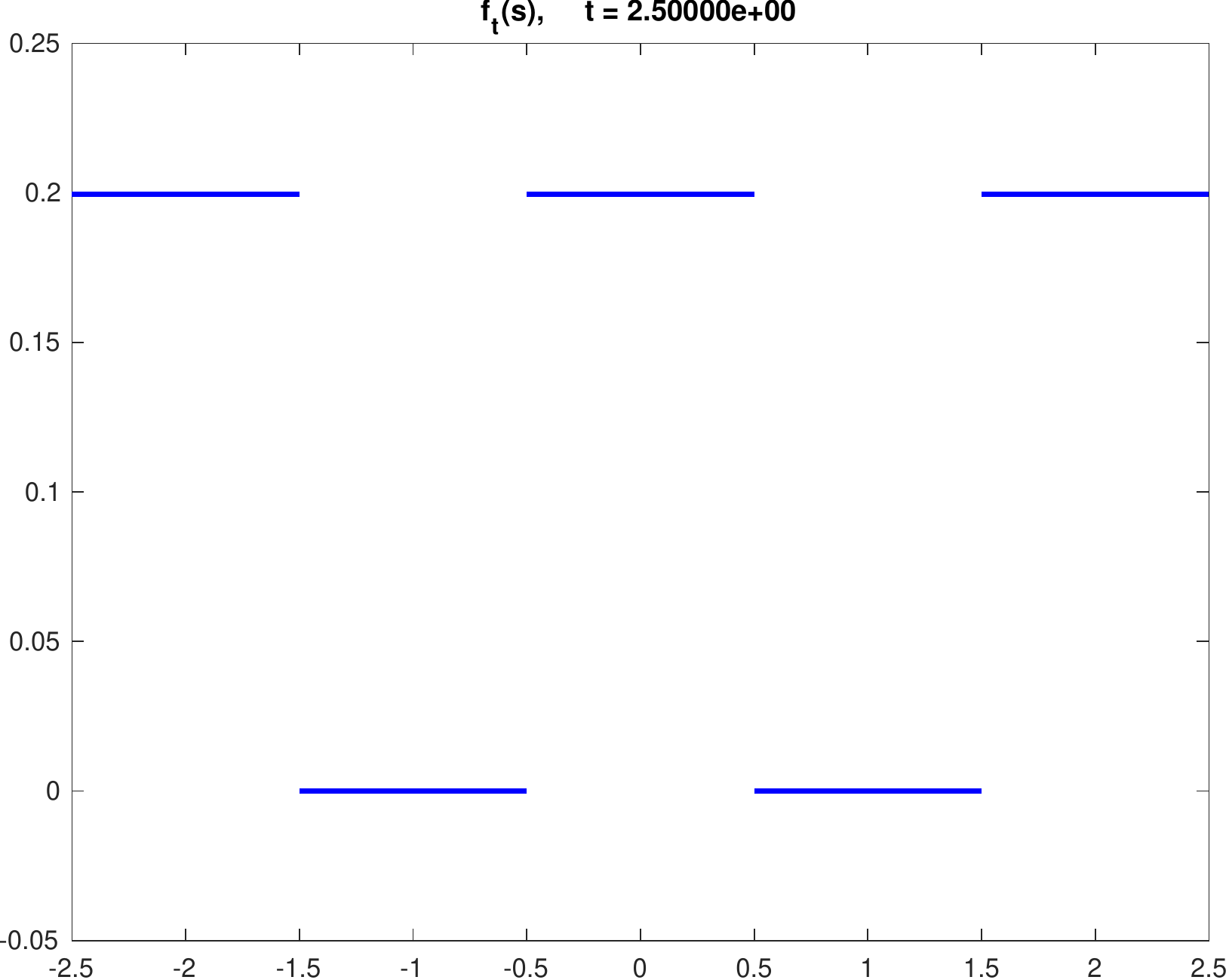}
\includegraphics[width=0.3\textwidth]{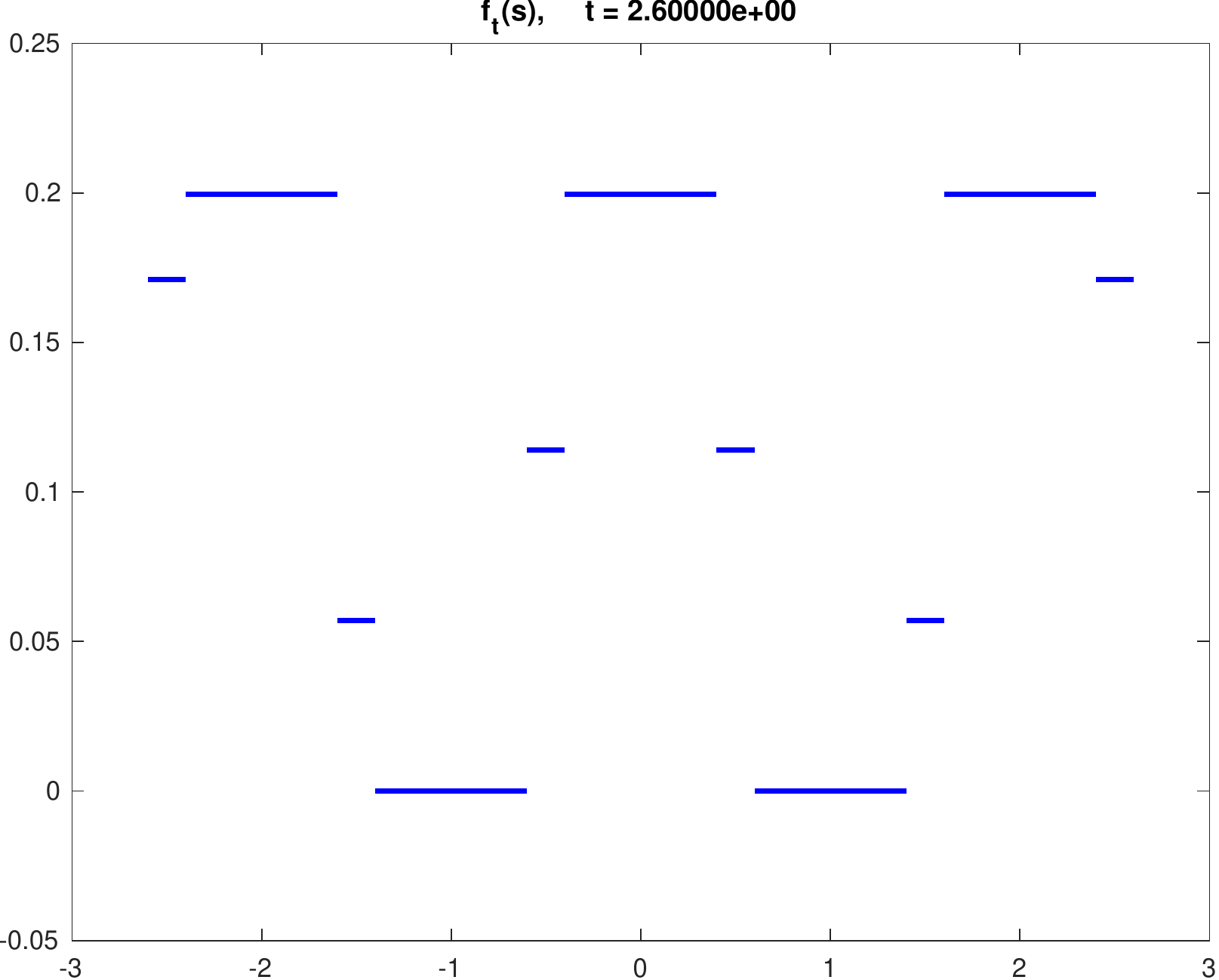}
\includegraphics[width=0.3\textwidth]{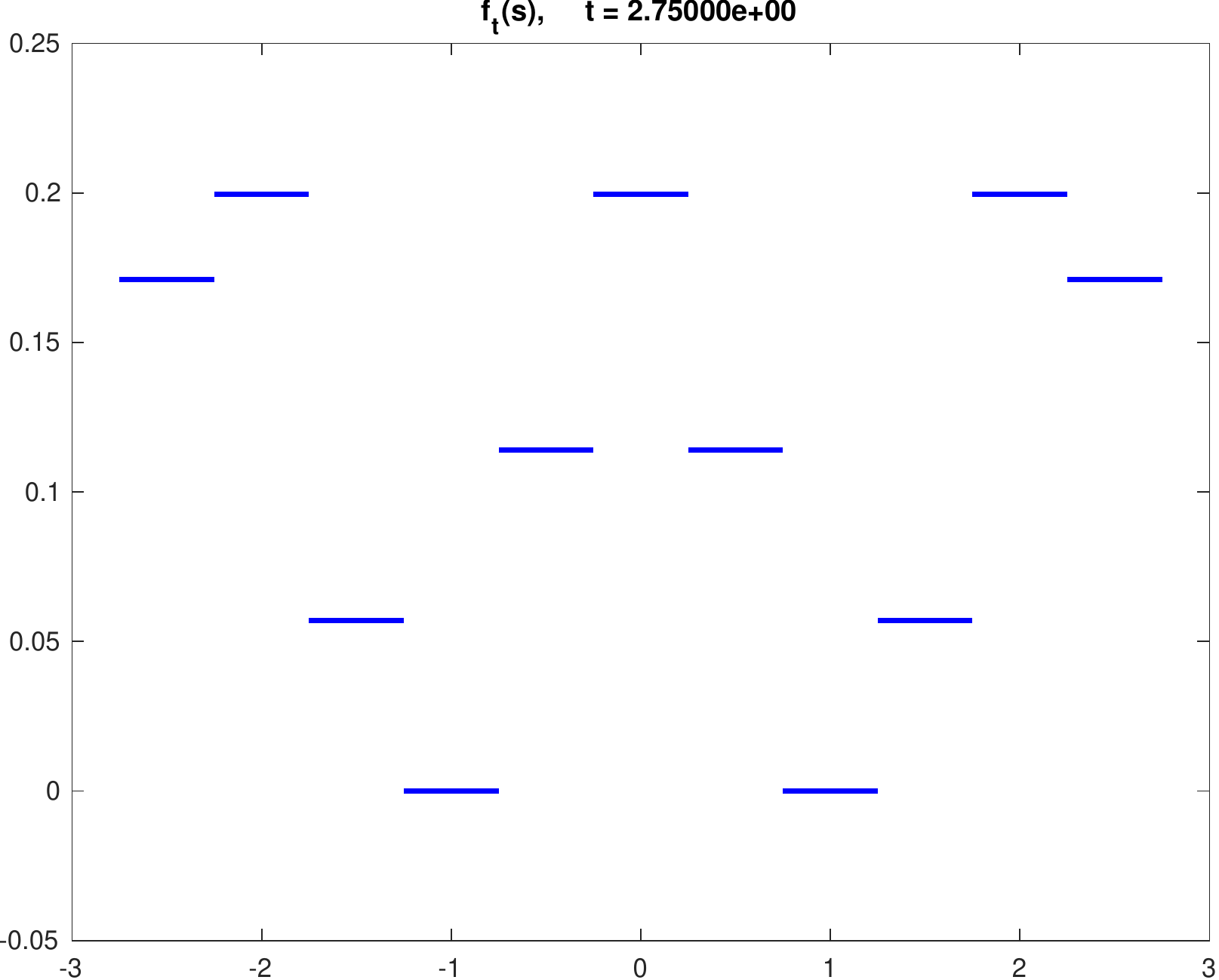}
\includegraphics[width=0.3\textwidth]{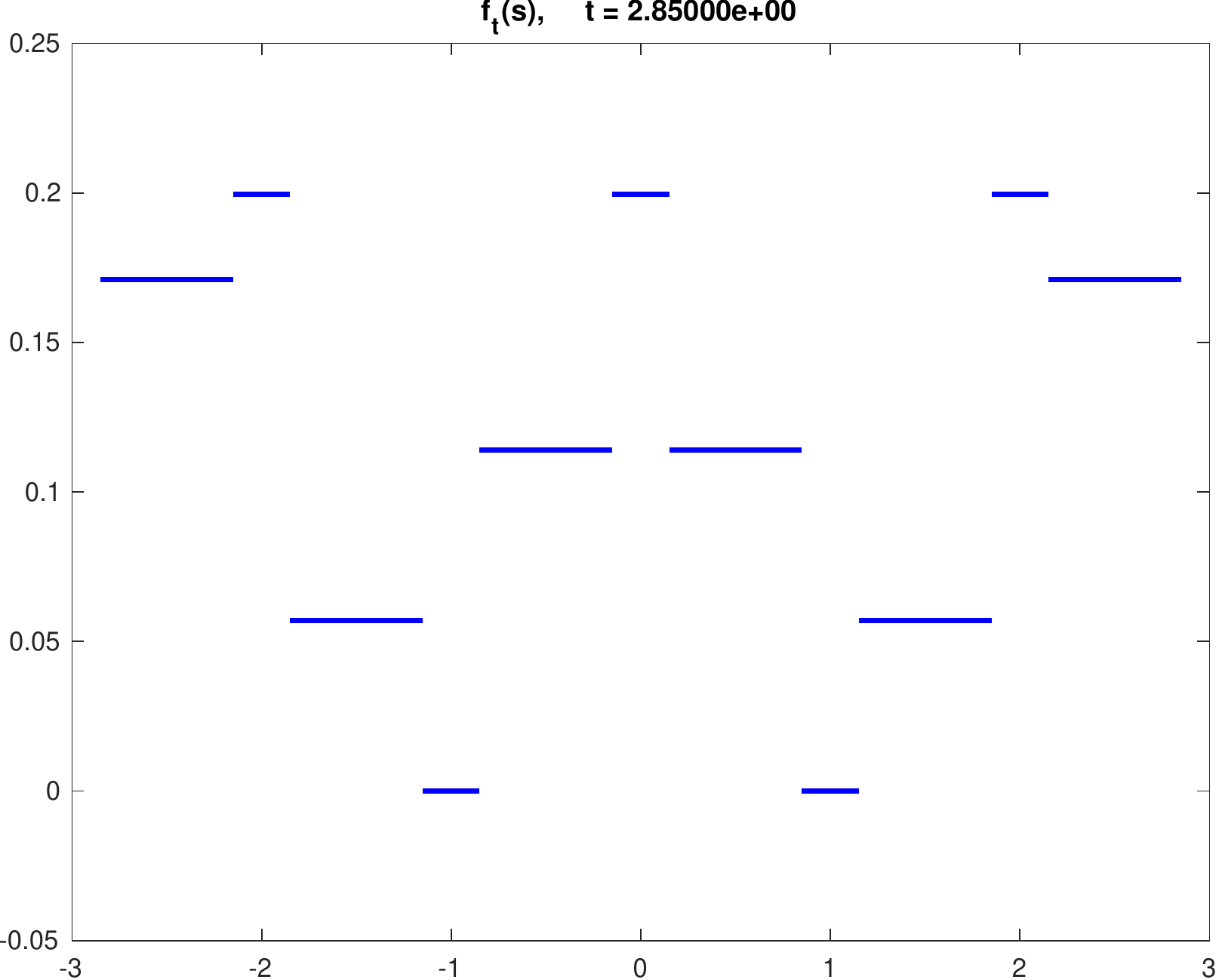}
\includegraphics[width=0.3\textwidth]{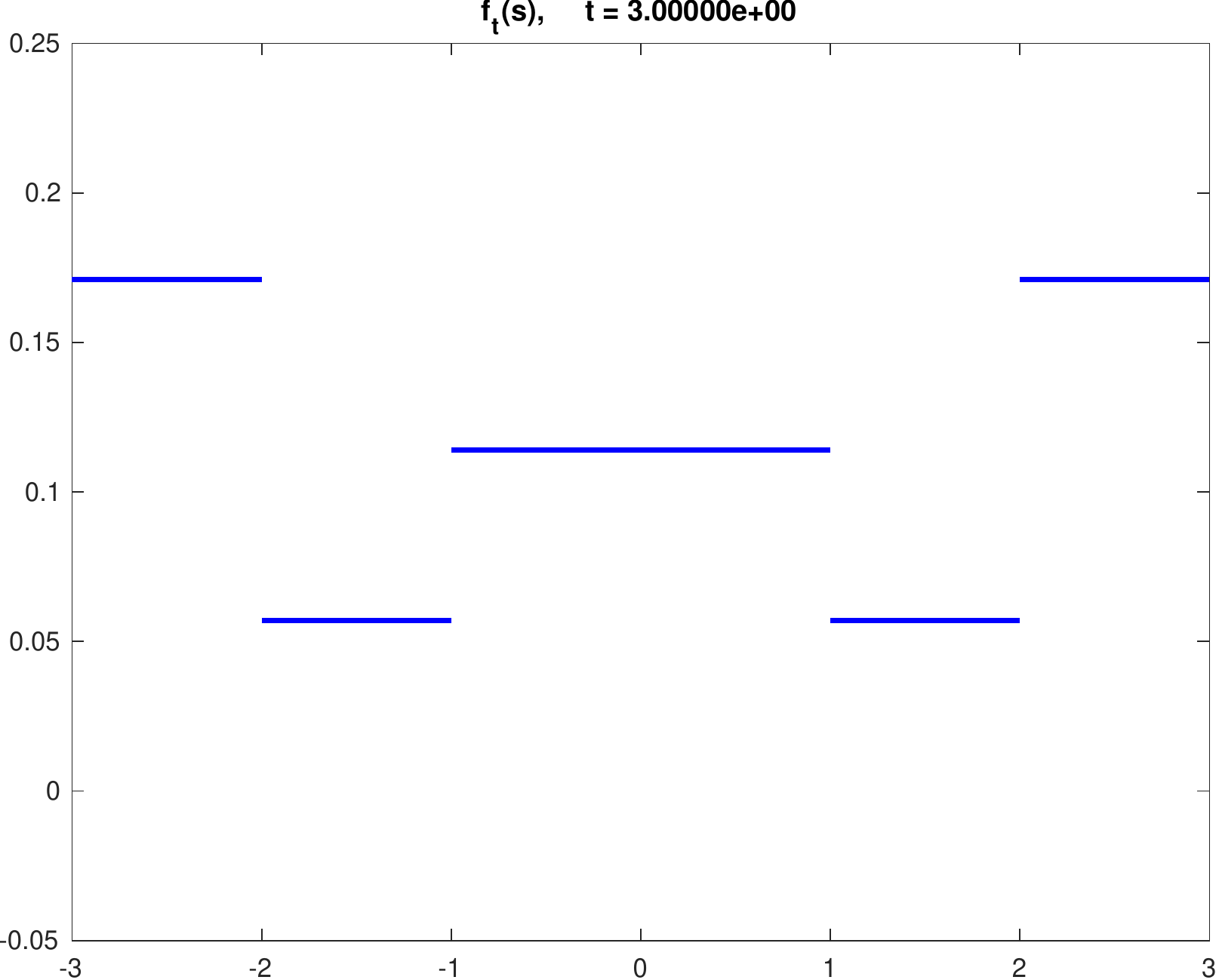}
\caption{First row: left $t = 2.5$, middle $t = 2.6$, right $t = 2.75$. Second row: left $t = 2.85$, right $t = 3$.}
\end{figure}

Based on the algorithm, the corresponding $h_{11}(t)$ is a step function of step size $\frac{1}{2}$, and the steps take value \begin{equation*}
    1, \frac{1}{3}, \frac{2}{3}, \frac{2}{5}, \frac{3}{5}, \ldots, \frac{n}{2n - 1}, \frac{n}{2n + 1}, \frac{n + 1}{2n + 1}, \ldots
\end{equation*}
on intervals 
\begin{equation*}
    \left[0, \frac{1}{2}\right), \left[ \frac{1}{2}, 1 \right), \left[ 1, \frac{2}{3} \right), \ldots
\end{equation*}
\end{example}

Via a change of variable one can extend the last statement to $2T-$periodic measures:

\begin{corollary}\label{2Talgorithm}
Let $\mu$ be an even $2T$-periodic PW-measure, $d\mu(x) = \left( \sum_{n = 0}^\infty a_n cos\left(\frac{2 n \pi}{2T} x\right) \right)dx$. Consider the infinite Toeplitz matrix \begin{equation*}
    J = \begin{pmatrix} a_0 & \frac{a_1}{2} & \frac{a_2}{2} & \frac{a_3}{2} & \frac{a_4}{2} & \ldots\\
    \frac{a_1}{2} & a_0 & \frac{a_1}{2} & \frac{a_2}{2} & \frac{a_3}{2} & \ldots\\
    \frac{a_2}{2} & \frac{a_1}{2} & a_0 & \frac{a_1}{2} & \frac{a_2}{2} & \ldots\\
    \frac{a_3}{2} & \frac{a_2}{2} & \frac{a_1}{2} & a_0 & \frac{a_1}{2}  & \ldots\\
    \vdots & \vdots & \vdots & \vdots & \vdots  & \ddots\\
    \end{pmatrix}.
\end{equation*}

Denote by $J_n$ the $(n + 1) \times (n + 1)$ sub-matrix on the upper-left corner of $J$. Then $h_{11}(t)$ in the Hamiltonian is a step function with \begin{equation*}
    h_{11}(t) = \Sigma(J_n^{-1}) - \Sigma(J_{n - 1}^{-1}) \; \text{on} \; \Big(\frac{n\pi}{2T}, \frac{(n + 1)\pi}{2T}\Big], n = 0, 1, 2, \ldots
\end{equation*}
where $\Sigma$ denotes the sum of all elements in the matrix, and $J_{-1}^{-1} = 0$.
\end{corollary}

A $2\pi$-periodic measure can be naturally identified with a measure on the unit circle $\mathbb{T}$ in the complex plane.
The inverse spectral problem for canonical systems in the case of a periodic spectral measure is closely related to the theory
of orthogonal polynomials on $\mathbb{T}$. The following theorem illustrates this connection.

\begin{theorem}[\cite{MP}]\label{onpoly}
The value of the $n$-th step of $h_{11}(t)$ is $\varphi_n(1)^2$, where $\varphi_n(z)$ is the $n$-th orthonormal polynomial on $\mathbb{T}$ with $\mu_\mathbb{T} = \frac{1}{2\pi} \mu|_{[-\pi, \pi]}$.
\end{theorem}

The theorem above allows one to solve inverse spectral problems for measures $\mu$ whose orthonormal polynoials are known, without additional calculations. Tables of orthogonal polynomials for most common measures can be found in the literature, see \cite{Simon}
for examples and further references.

\begin{remark}
In addition to the direct proof presented in \cite{MP}, the theorem can also be obtained algebraically from the Heine formulae. The Heine formulae can be found in books on orthogonal polynomials, see for example \cite{Simon}.

Heine formulae for $\Phi_n$ (the $n$-th monic polynomial on $\mathbb{T}$, with $d\mu_{\mathbb{T}} = \sum_{n = 0}^\infty a_n \cos(nx)\frac{dx}{2\pi}$)
\begin{equation*}
\begin{split}
    &\Phi_n(z) = \frac{1}{\det(J_{n - 1})}  \begin{vmatrix}
    a_0 & \frac{a_1}{2} & \frac{a_2}{2} & \ldots & \frac{a_n}{2} \\ \frac{a_1}{2} & a_0 & \frac{a_1}{2} & \ldots & \frac{a_{n - 1}}{2}\\ \frac{a_2}{2} & \frac{a_1}{2} & a_0 & \ldots & \frac{a_{n - 2}}{2} \\ \vdots & \vdots & \vdots & \ddots & \vdots \\ 1 & z & z^2 & \ldots & z^n
    \end{vmatrix},\\
    &\|\Phi_n(z)\|^2 = \frac{\det(J_{n})}{\det(J_{n - 1})}
\end{split}    
\end{equation*}

Like before, we use $\varphi_n(z)$ to denote the $n-$th orthonormal polynomial on $\mathbb{T}$. Then by the formulae above, we have \begin{equation*}
    \varphi_n(1)^2 = \frac{1}{\det(J_{n - 1})\det(J_{n})} \begin{vmatrix}
    a_0 & \frac{a_1}{2} & \frac{a_2}{2} & \ldots & \frac{a_n}{2} \\ \frac{a_1}{2} & a_0 & \frac{a_1}{2} & \ldots & \frac{a_{n - 1}}{2}\\ \frac{a_2}{2} & \frac{a_1}{2} & a_0 & \ldots & \frac{a_{n - 2}}{2} \\ \vdots & \vdots & \vdots & \ddots & \vdots \\ 1 & 1 & 1 & \ldots & 1
    \end{vmatrix}^2,
\end{equation*}
which equals to $\Sigma(J_n^{-1}) - \Sigma(J_{n - 1}^{-1})$.  To verify this, one can, for instance, write $J_{n - 1}^{-1}$ and $J_n^{-1}$ using their co-factor matrices, so that $\Sigma(J_n^{-1}) - \Sigma(J_{n - 1}^{-1})$ could be written down explicitly. For $\varphi(1)^2$, one can then expand the determinant using the last row.
    
\end{remark}

\section{Inverse spectral problems using periodization}\label{Periodization}

Let $\mu$ be a  measure on $\R$. We denote by $\mu_T$ its $2T-$periodization, defined as $\mu_T(S) = \mu(S)$ for $S\subseteq [-T, T]$, and  extended periodically to the rest of $\R$.

By corollary \ref{PWper}, a locally finite periodic measure is a PW-measure if and only if it has infinite support on its period. If a locally finite $\mu$ has infinite support on an interval $[-C, C]$ for some $C > 0$, then $\mu_T$ is a PW-measure for $T > C$. Therefore, we can recover a PW-system corresponding to $\mu_T$ using corollary \ref{2Talgorithm} or theorem \ref{onpoly}. 

{\it Throughout the rest of the paper we denote
by $h_{11}$ and $h_{11}^T$ the  upper-left entries of the Hamiltonians corresponding to  $\mu$ and $\mu_T$ respectively.}

Our goal in this section is to show that $h_{11}^T$ converges to $h_{11}$ as $T\to\infty$ in various classes of spectral measures $\mu$.

Recall that $K_0^t$ are the reproducing kernels in the de Branges' spaces $B_t(\mu)$. Denote by $K_0^{t, T}$ the reproducing kernels in the de Branges' spaces $B_t(\mu_T)$. We begin with the following lemma on the convergence of $h_{11}^T$ to $h_{11}$.

\begin{lemma}\label{SuffCondition}
If $ \lim_{T \to \infty} K_0^{t, T}(0) = K_0^t(0)$ for almost every $t$, then for any
interval $(a,b)\subset \R_+$, 
$$\int_a^b h_{11}^T(t)dt  \to \int_a^b h_{11}(t)dt $$  as $T \to \infty$.
\end{lemma}

\begin{proof} Recall that 
$$\int_a^b h_{11}^T(t)dt = \pi \left( K_0^{b, T}(0) - K_0^{a, T}(0) \right)$$
and
$$\int_a^b h_{11}(t)dt = \pi \left( K_0^{b}(0) - K_0^{a}(0) \right).$$
\end{proof}

Since the Hamiltonian matrix $H(t) \geqslant 0$ almost everywhere, $h_{11} \geqslant 0$ almost everywhere.
For non-negative functions convergence of integrals over intervals, like in the last lemma, 
is equivalent to the convergence on continuous compactly supported functions:

\begin{corollary}
If $ \lim_{T \to \infty} K_0^{t, T}(0) = K_0^t(0)$ for almost every $t$, then
$$\int_0^\infty h_{11}^T(t)\phi(t)dt{\to} \int_0^\infty h_{11}(t)\phi(t)dt$$
as $T\to\infty$, for every continuous compactly supported function $\phi$ on $\R_+$.
\end{corollary}

\begin{proof}
Let $\text{supp} (\phi) \subseteq [a,b]$.
Since $h_{11}^T\geq 0 $
and $\int_a^b h_{11}$ converge, 
$h_{11}^T$ are bounded in $L^1([a,b])$-norm. Therefore, from
any sequence $h_{11}^{T_n},\ T_n\to\infty$,
one can choose a subsequence $h_{11}^{T_{n_k}}$ such that the measures
$\mu_k,\ d\mu_k(t)=h_{11}^{T_{n_k}}(t)dt$ converge in the
$*$-weak topology of the space of finite measures on $[a,b]$ to a finite non-negative measure
$\mu$. If $d\mu\neq h_{11}(t)dt$ then
there exists a subinterval $[c,d]\subseteq [a,b]$ such that 
$$\mu([c,d])=\lim\int_c^dh_{11}^{T_{n_k}}(t)dt\neq \int_c^d h_{11}(t)dt,$$
which contradicts to the last lemma. 
Therefore, $d\mu= h_{11}(t)dt$. We obtain
that any partial limit of $h_{11}^T$
in the $*$-weak topology of the space of measures on $[a,b]$ is equal to $h_{11}$, which implies that $h_{11}^T\to h_{11}$ in the 
$*$-weak topology, i.e., on every continuous function on $[a,b]$ like in the statement.
\end{proof}

Let  $f_T, T>0$
and $f$ be $L^1_{loc}(\R_+)$-functions. We write $f_T \overset{*}{\to} f$ as $T \to \infty$ if $f_T, T>0$
and $f$ satisfy the conclusions of
the last lemma and corollary in place of $h_{11}^T$ and $h_{11}$. Using this
notation, $ \lim_{T \to \infty} K_0^{t, T}(0) = K_0^t(0)$ for almost every $t$ implies $h_{11}^T\overset{*}{\to} h_{11}$
as $T\to\infty$.

\subsection{Convergence for \texorpdfstring{$PW-$}{PW}systems}
First, we study convergence of periodizations in the class of canonical systems whose spectral measures are from the PW-class. 
As was mentioned before, it extends the class of regular systems considered in classical literature and was recently 
studied in \cite{BR, Bessonov, MP}. We will need the following

\begin{lemma}\label{uniformlb}
Let $\mu$ be a PW-measure, and $\mu_T$ be the $2T-$periodic extensions of $\mu$. For a fixed $t>0$, there are constants $c_t,C_t > 0$ such that 
$$c_t \norm{f}_{PW_t} \leqslant \norm{f}_{L^2(\mu_T)}\leqslant C_t\norm{f}_{PW_t}$$ for all $f \in PW_t$ and $T$ sufficiently large.
\end{lemma}

\begin{proof}
By theorem \ref{PWcondition}, for some $d > t$, there exist $\delta=\delta_\mu > 0$ and $L=I_\mu>0$, such that for all intervals $I$ with $\abs{I} > I_\mu$ there exists at least $d \abs{I}$ disjoint $(\mu, \delta_\mu)-$intervals intersecting $I$. 
By results of \cite{OS}, the inequality 
$$c_t \norm{f}_{PW_t} \leqslant \norm{f}_{L^2(\nu)}\leqslant C_t\norm{f}_{PW_t}$$ holds for all
$f\in PW_t$ for any measure $\nu$ satisfying this condition with the same $d,\delta$ and $L$. 
One can show that the constants $c_t, C_t$ can be chosen to be the same for any such measure.
It is left to notice that $\mu_T$ becomes such a measure for $ T > \frac{I_\mu}{2}$. 
\end{proof}

For $PW-$systems we prove the following convergence result. 

\begin{theorem} \label{PW}
Let $\mu$ be an even $PW$-measure with locally infinite support. Then
\begin{equation*}
    h^T_{11}\overset{\ast}{\to} h_{11}\text{ as }T\to\infty.
\end{equation*}
\end{theorem}

\begin{proof}
As was discussed before, periodizations of $\mu$ are PW-measures for $T$ large enough. Therefore, we can apply corollary \ref{2Talgorithm} to $\mu_T$. By lemma \ref{SuffCondition}, it suffices to prove that for every $t$, $K_0^{t, T}(0) \to K_0^t(0)$ as $T \to \infty$. With \eqref{h11f}, it suffices to show that \begin{equation*}
    \int_{-t}^t f_t^T(s) ds \to \int_{-t}^t f_t(s) ds, \quad \text{as} \; T \to \infty.
\end{equation*}

For every fixed $t > 0$, recall that $f_t^T \ast \hat{\mu}_T (x) = f_t \ast \hat{\mu} (x) = 1$ for $x \in [-t, t]$, where $f_t^T$ and $f_t$ are the Fourier transforms of $K_0^{t, T}$ and $K_0^t$ respectively. Then for $x \in [-t, t]$, we have \begin{equation}
\begin{split}
    0 &= f_t^T \ast \hat{\mu}_T (x) - f_t \ast \hat{\mu} (x)\\
    &= \left( f_t^T - f_t \right) \ast \hat{\mu}_T(x) + f_t \ast (\hat{\mu}_T - \hat{\mu})(x).
\end{split}\label{eq002}
\end{equation}

Let $\phi\geq 0 $ be a $C^\infty$-function supported in $[-2t,2t]$. Then $\check\phi$ is a fast
decreasing function so that
$$c_n=\int_n^{n+1}|\check\phi(x)|dx$$
satisfy $c_n=o(1/n^3)$ as $|n|\to\infty$.
Recalling that $f_t$ is the 
Fourier transform of $K_0^t$,  we have 
\begin{equation}
\phi \cdot (f_t \ast (\hat{\mu}_T - \hat{\mu}))(x) = \frac{1}{\sqrt{2\pi}} \int_{-\infty}^\infty 
e^{-i \pi x s}\left[\int_{-\infty}^\infty \check\phi(s-y)K_0^t(y)  d(\mu_T - \mu)(y)\right]ds.\label{eq001}
\end{equation}

Let
$$G(s) = \int_{-\infty}^\infty \check\phi(s - y) K_0^t(y)  d(\mu_T - \mu)(y).$$
Since $h(y)=K_0^t(y)$ belongs to $PW_t$, its derivative
$h'(y)$ also belongs to $PW_t$ and therefore belongs to $L^2(\R)$. Let
$$d_n=\int_n^{n+1}|h'(y)|dy,\ n\in\mathbb Z.$$
By the Cauchy-Riemann inequality, $\{d_n\}\in l^2$. Also, since $\mu$ is PW-measure, $\mu([n,n+1])<C$ for all $n$ by theorem \ref{PWcondition},
and therefore $\mu_T([n,n+1])<2C$.

For $s\in[N,N+1]$,
\begin{equation*}
\begin{split}
    |G(s)| &\leqslant \int_{-\infty}^\infty \abs{\check\phi(s - y)}\abs{h(y)}  d\abs{\mu_T - \mu}(y)\\
    &= \sum_{n\in\mathbb Z}\int_{s+n}^{s+n+1} \abs{\check\phi(s - y)} \abs{h(y)}  d\abs{\mu_T - \mu}(y)\\
    &\leqslant 6C\sum_{n\in\mathbb Z} c_n\sup_{y\in[N+n,N+n+2]}\abs{h(y)}\\
    &\leqslant 6C\sum_{n\in\mathbb Z} c_n\left(\inf_{y\in[N+n,N+n+2]}|h(y)|+d_{N+n}+d_{N+n+2}\right)=p_N
    \end{split}
\end{equation*}
Since $h\in L^2(\R)$, $\inf_{y\in[n,n+2]}|h(y)|$ is an $l^2$-sequence. Since $\{d_n\}\in l^2$
and $c_n\in l^1$, $p_n$ is a convolution of of an $l^2$-sequence and an $l^1$-sequence, and therefore is itself an $l^2$-sequence.

Also notice that the fast decay of $\check\phi$, boundedness of $h$ as a PW-function and
the property that $\mu_T-\mu\equiv 0$ on $[-T,T]$ imply that for $s\in [-T/2,T/2]$,
\begin{equation*}
    |G(s)|\lesssim \frac 1{T^2}. 
\end{equation*}
All in all,
$$||G||_{L^2(\R)}=\int_{|x|>T/2}|G(x)|^2dx + \int_{|x|<T/2}|G(x)|^2dx
\leqslant \sum_{|n|>\frac T2-1}p_n^2 +O(1/T)\to 0\ \text{ as }T\to\infty.$$
Therefore the left-hand side in \eqref{eq001}, which equals to $\hat G$,
 tends to zero in $L^2(\R)$ (as a function of $x$).

Since $\phi$ was an arbitrary test function supported in $[-2t,2t]$, $f_t \ast (\hat{\mu}_t - \hat{\mu})$ converges to $0$ in $L^2([-t, t])$ as $T \to \infty$. By \eqref{eq002} it follows that $\left( f_t^T - f_t \right) \ast \hat{\mu}_T$ also converges to $0$ in $L^2([-t, t])$ as $T \to \infty$. Lemma \ref{uniformlb} implies that $L_{\mu_T}$ is invertible for large enough $T$, and $\|L_{\mu_T}^{-1}\|$ is uniformly bounded above by $c_t$. Therefore,  \begin{equation*}\begin{split}
    \norm{f_t^T - f_t}_{L^2([-t, t])} = \norm{\mathcal F\left(f_t^T - f_t\right)}_{PW_t}=
    \norm{L_{\mu_T}^{-1}\left(\mathcal F\left(\left( f_t^T - f_t \right) \ast \hat{\mu}_T\right)\right)}_{PW_t}\\ \leqslant c_t \cdot ||\left( f_t^T - f_t \right) \ast \hat{\mu}_T||_{L^2([-t, t])}\to 0,
\end{split}\end{equation*}
as $T\to\infty$.
Thus, for every fixed $t$ we have $f_t^T \to f_t$ in $L^2([-t, t])$, which implies \begin{equation*}
    \int_{-t}^t f_t^T(s) ds \to \int_{-t}^t f_t(s) ds, \quad \text{as} \; T \to \infty.
\end{equation*}
\end{proof}

Until recently, examples of solutions to inverse spectral problems for canonical systems were rare in the literature. Even though the number of available examples has increased with the development of new methods, 
most of the new examples pertain to periodic spectral measures. Non-periodic 
examples are still difficult to calculate explicitly, even for the simple spectral measures. 
The results of this section allow one to find such examples numerically after using the theorems of the 
previous section to find the Hamiltonians of the periodizations of the spectral measure.

\begin{example} Consider the PW-system with $d\mu(x) = (1 + \frac{\sin x}{x}) dx$. The density of this measure is bounded and bounded away from zero. Therefore $\mu$ is a PW-measure. By theorem \ref{PW}, we can use the Hamiltonians recovered from the periodizations $\mu_T$ to approximate the Hamiltonian corresponding to $\mu$. Here are the approximations with $T = \pi, 2\pi, 4\pi, 8\pi$. 

\begin{figure}[H]
\centering
\includegraphics[width=0.45\textwidth]{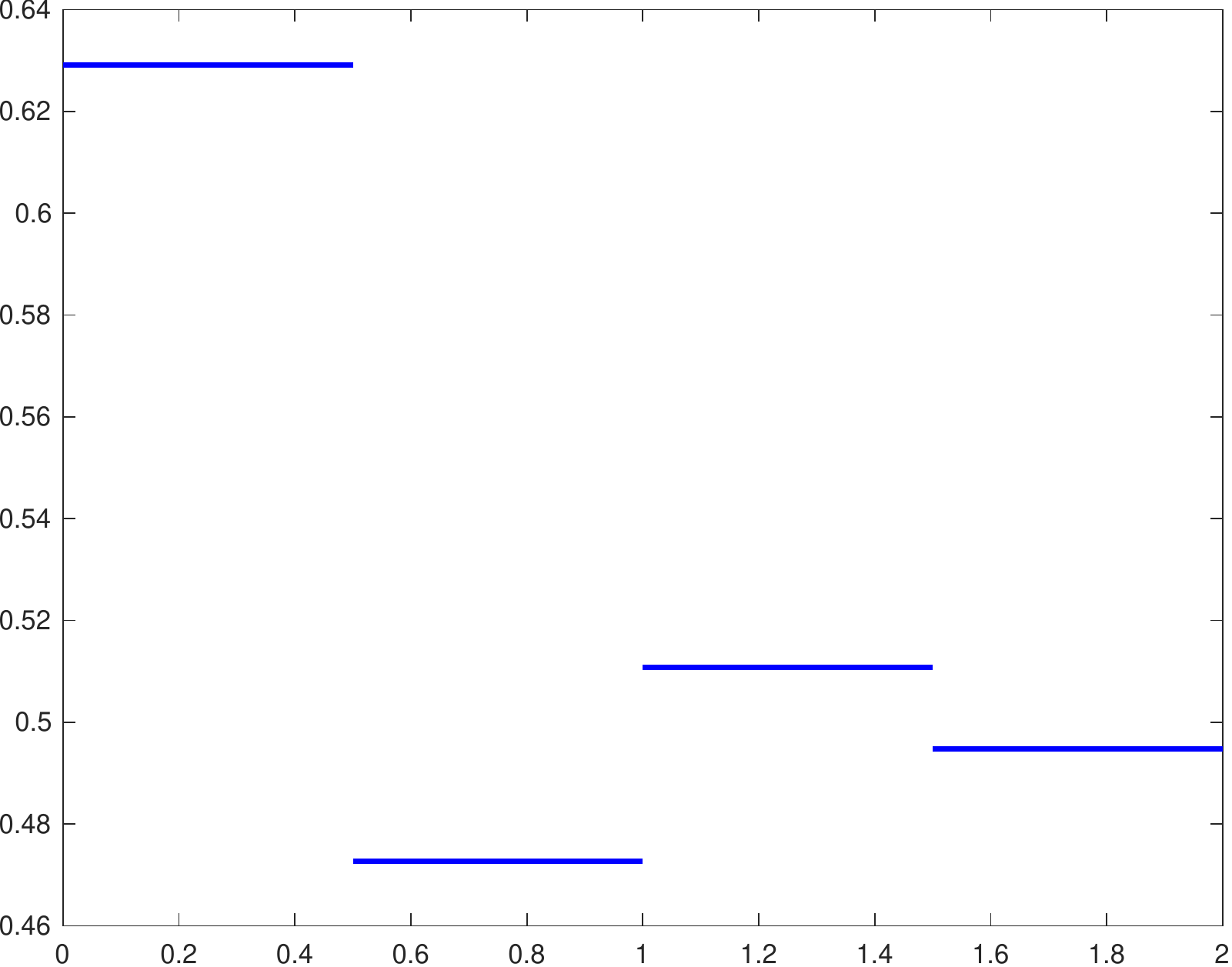}
\includegraphics[width=0.45\textwidth]{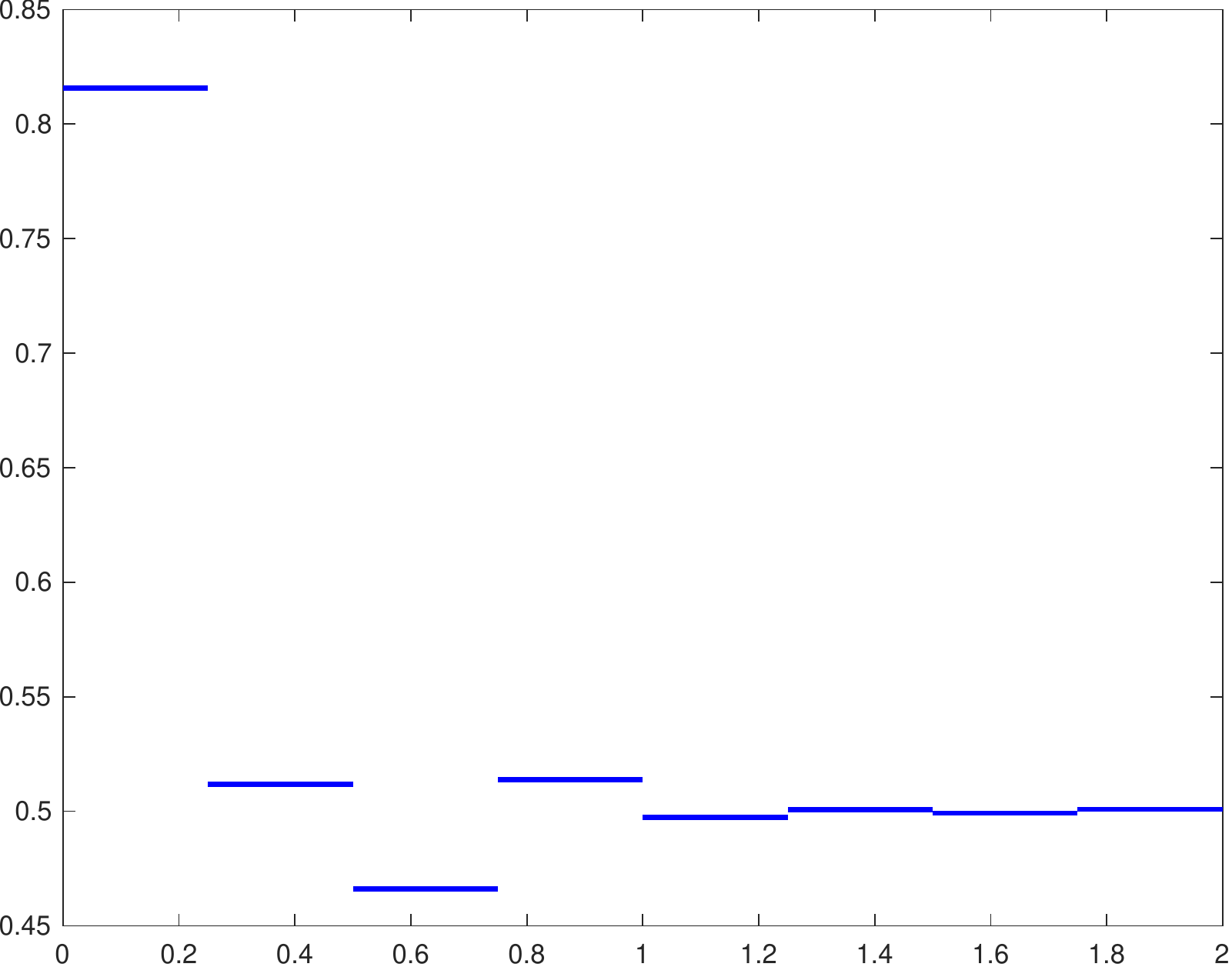}
\end{figure}

\begin{figure}[H]
\centering
\includegraphics[width=0.45\textwidth]{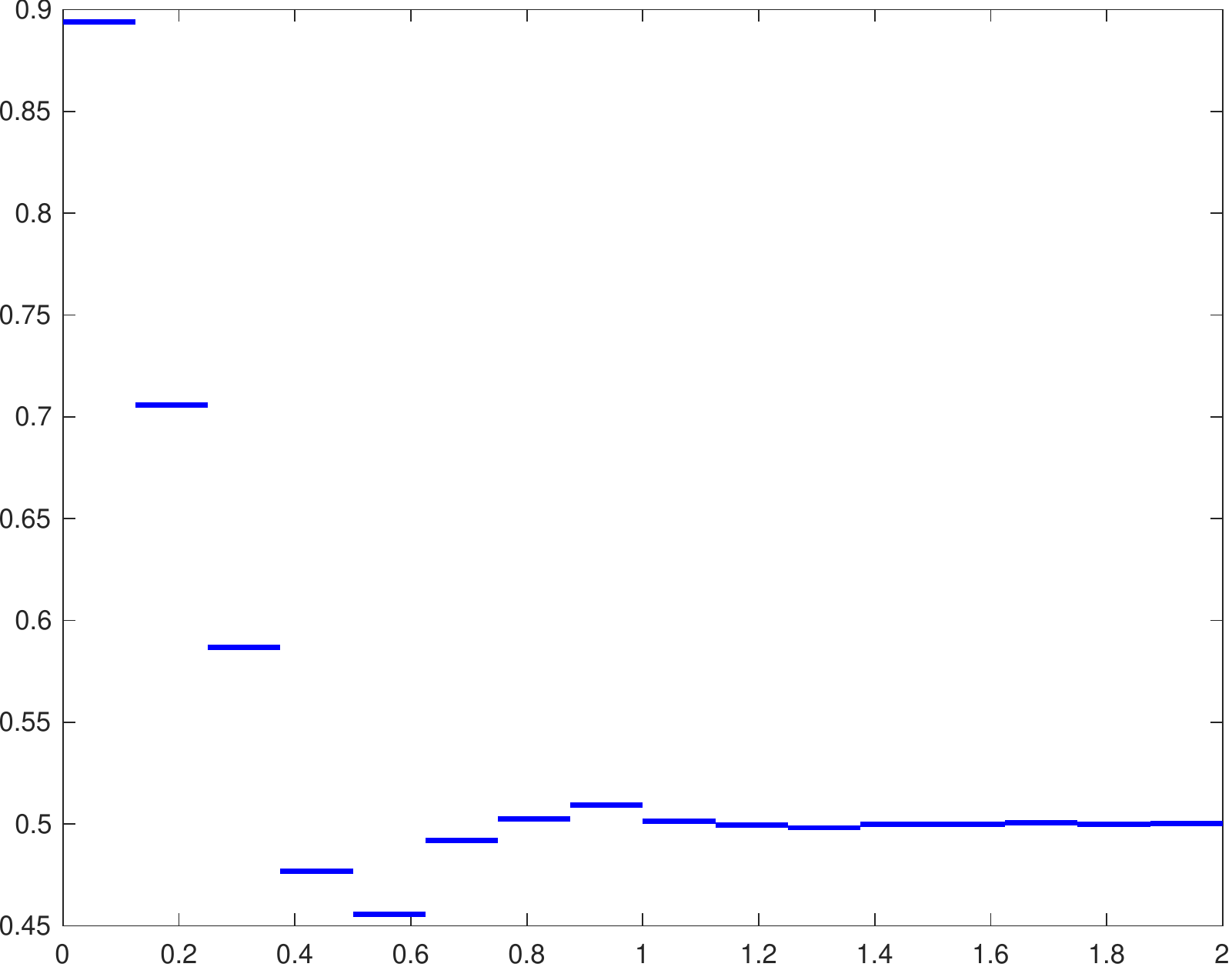}
\includegraphics[width=0.45\textwidth]{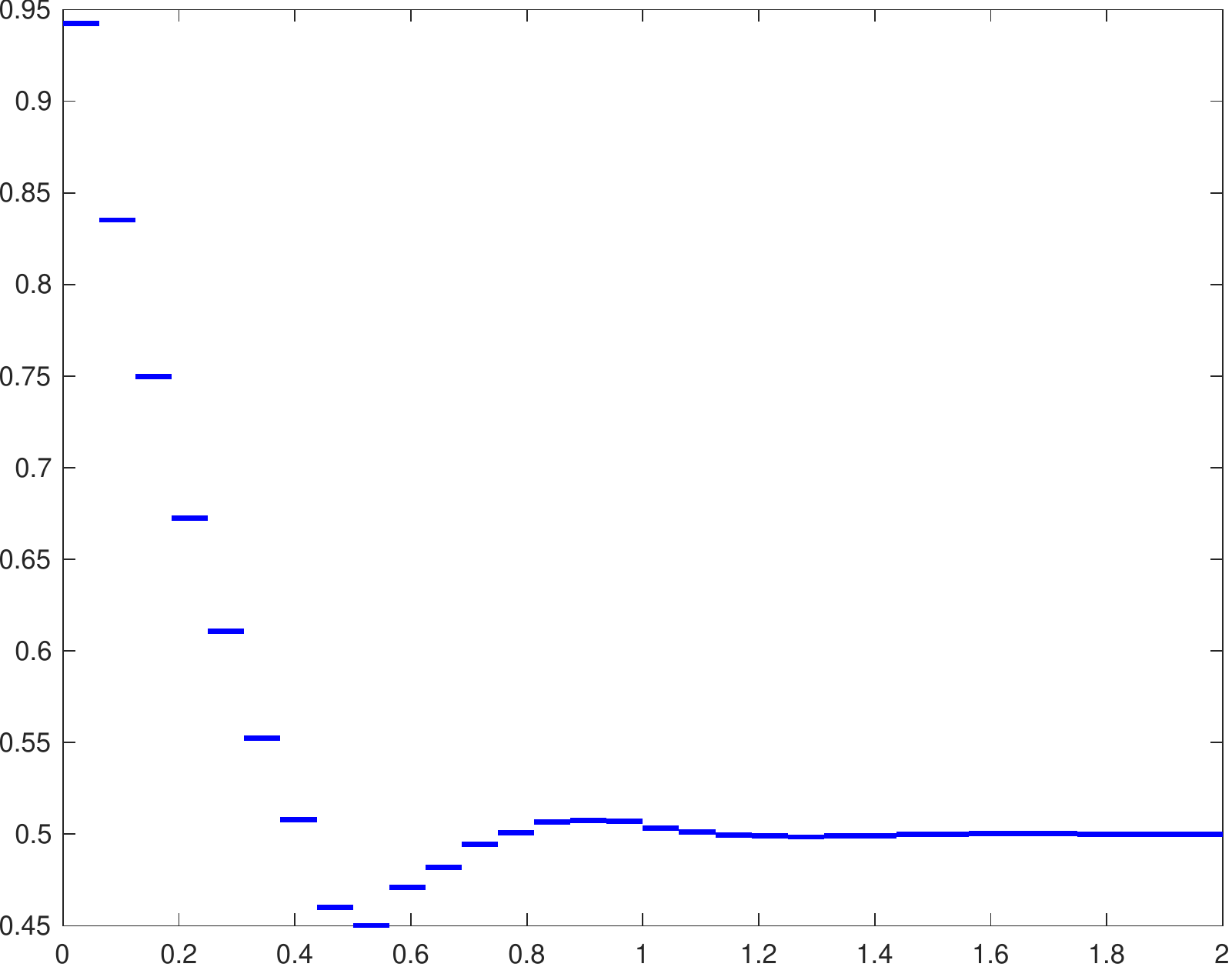}
\caption{Upper left: $h_{11}^{\pi}$. Upper right: $h_{11}^{2\pi}$. Lower left: $h_{11}^{4\pi}$. Lower right: $h_{11}^{8\pi}$.}
\end{figure}
\end{example}

\begin{example}
Consider the canonical system with $d\mu(x) = \left( 1 + \sin(x^2) \right) dx$. We use theorem \ref{PWcondition} to show that this is a PW-system.  For $x \in \mathbb{R}$, \begin{equation*}
   \mu((x, x + 1)) = \int_x^{x + 1} 1 + \sin(t^2) dt = \sqrt{\frac{\pi}{2}} \left( S\left(\sqrt{\frac{\pi}{2}} (x + 1)\right) - S\left(\sqrt{\frac{\pi}{2}} x\right) \right) + 1,
\end{equation*}
where $S$ is the Fresnel sine integral. The right-hand side is bounded, so $\sup_{x \in \mathbb{R}} \mu((x, x + 1)) < \infty$, which is the first condition in theorem \ref{PWcondition}. 

To check the second condition in theorem \ref{PWcondition}, it suffices to show that for every fixed $\epsilon$, there is a constant $c > 0$ such that $\mu(I) \geqslant c$ for all length $\epsilon$ intervals $I$. For every $x \in \mathbb{R}$ and $\epsilon > 0$, \begin{equation*}
    \mu((x, x + \epsilon)) = \int_x^{x + \epsilon} 1 + \sin(t^2) dt = \sqrt{\frac{\pi}{2}} \left( S\left(\sqrt{\frac{\pi}{2}} (x + \epsilon)\right) - S\left(\sqrt{\frac{\pi}{2}} x\right) \right) + \epsilon.
\end{equation*}
The right-hand side has a limit when $x \to \pm \infty$: \begin{equation*}
    \lim_{x \to \pm \infty}  \sqrt{\frac{\pi}{2}} \left( S\left(\sqrt{\frac{\pi}{2}} (x + \epsilon)\right) - S\left(\sqrt{\frac{\pi}{2}} x\right) \right) + \epsilon = \epsilon.
\end{equation*}
Therefore, we can find a constant $c > 0$ such that $\mu(I) \geqslant c$ for all length $\epsilon$ intervals $I$. By theorem \ref{PWcondition}, $\mu$ is a PW-measure, so the corresponding canonical system is a PW-system. By theorem \ref{PW}, we can use the Hamiltonians recovered from the periodizations $\mu_T$ to approximate the Hamiltonian corresponding to $\mu$. In particular, we look at $h_{11}^T$'s for $T = \pi, 2\pi, 4\pi, 8\pi$.

\begin{figure}[H]
\centering
\includegraphics[width=0.45\textwidth]{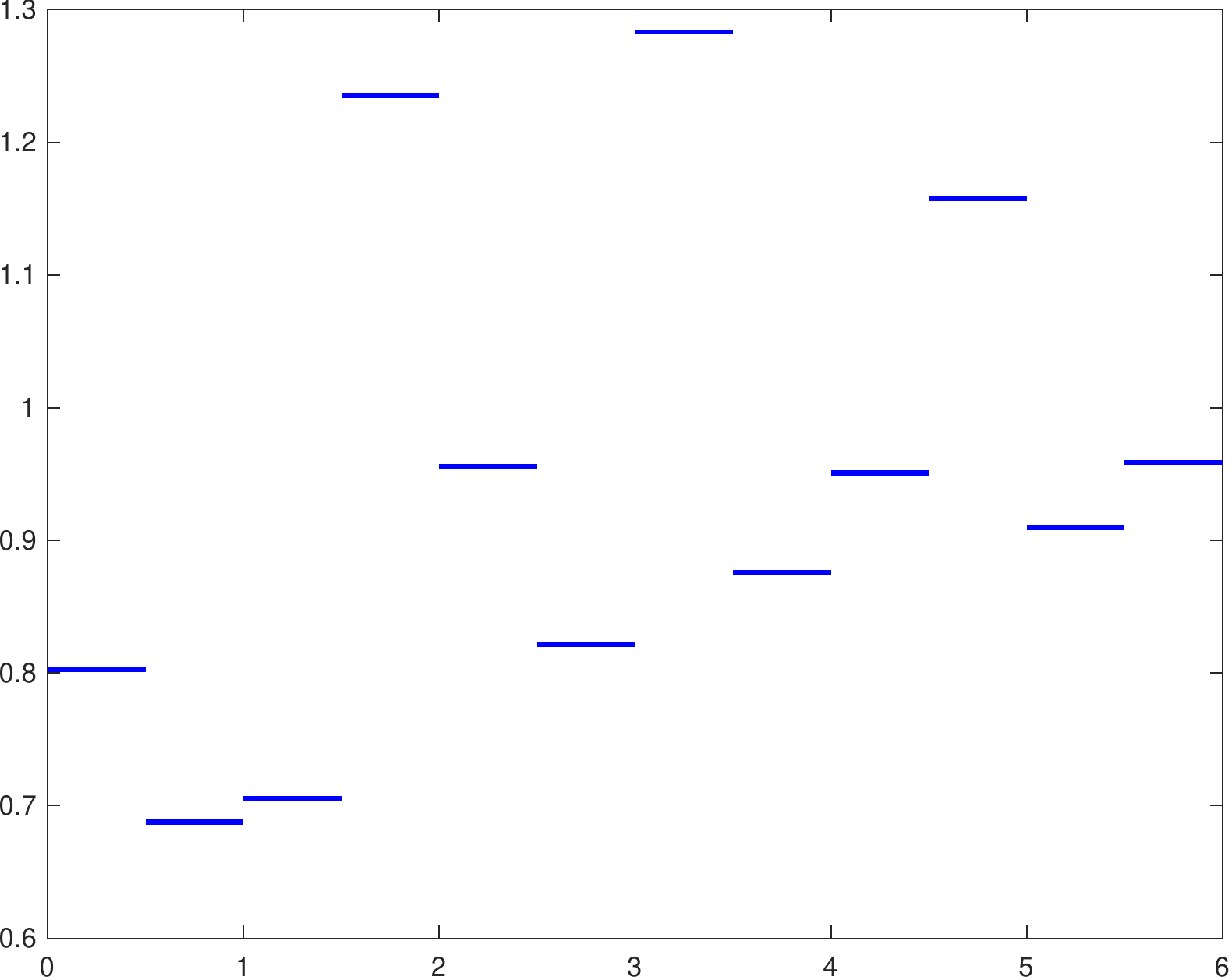}
\includegraphics[width=0.45\textwidth]{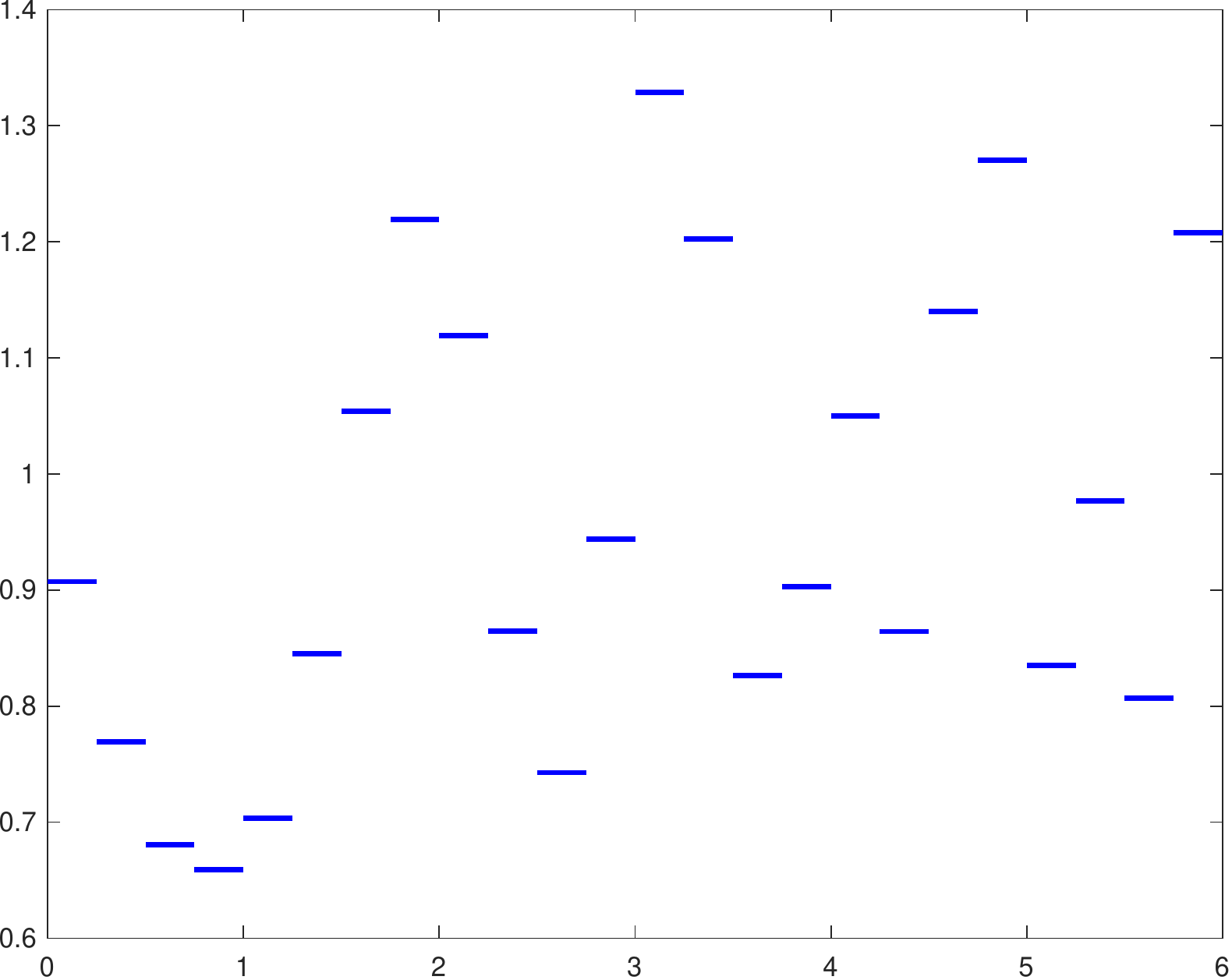}
\includegraphics[width=0.45\textwidth]{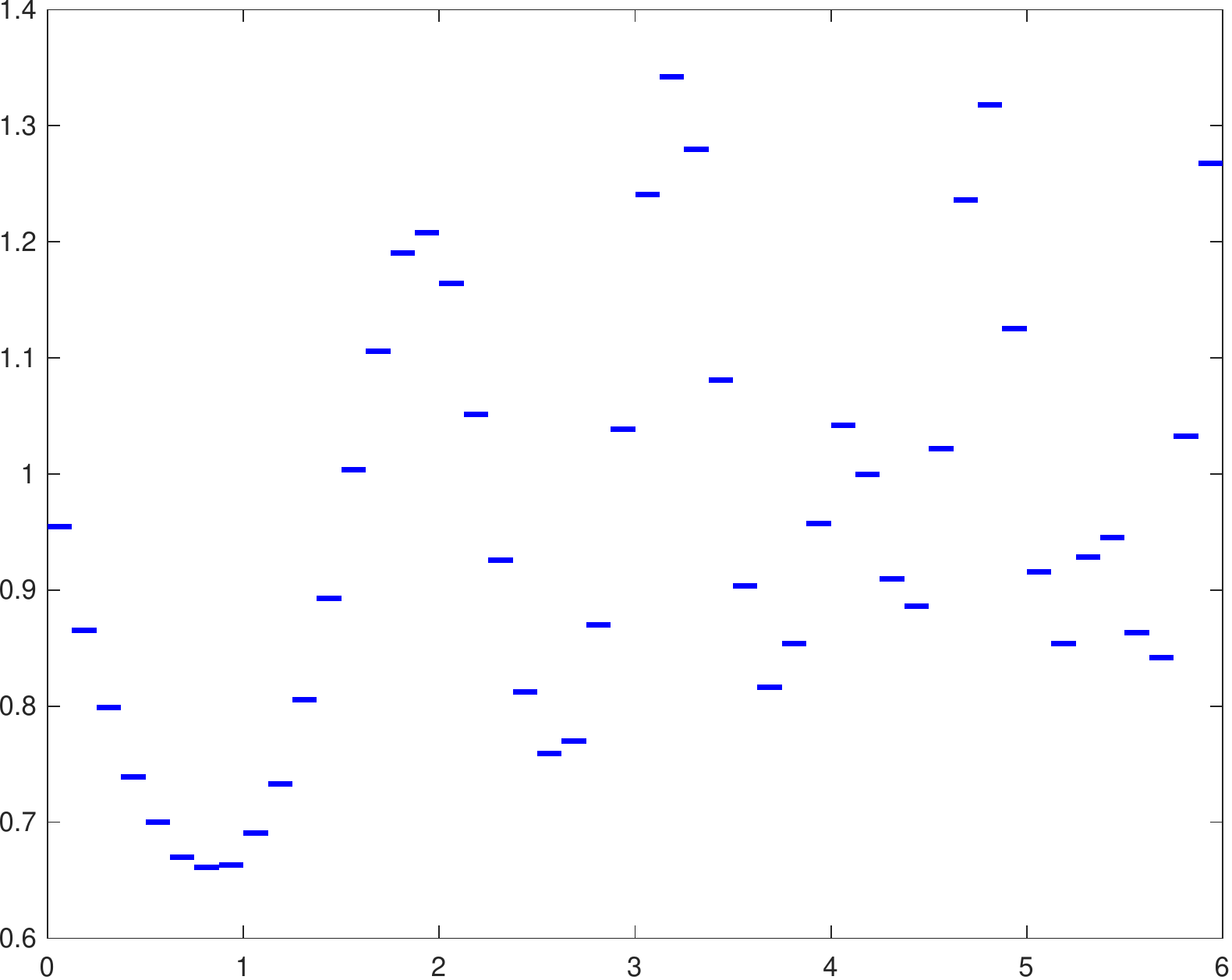}
\includegraphics[width=0.45\textwidth]{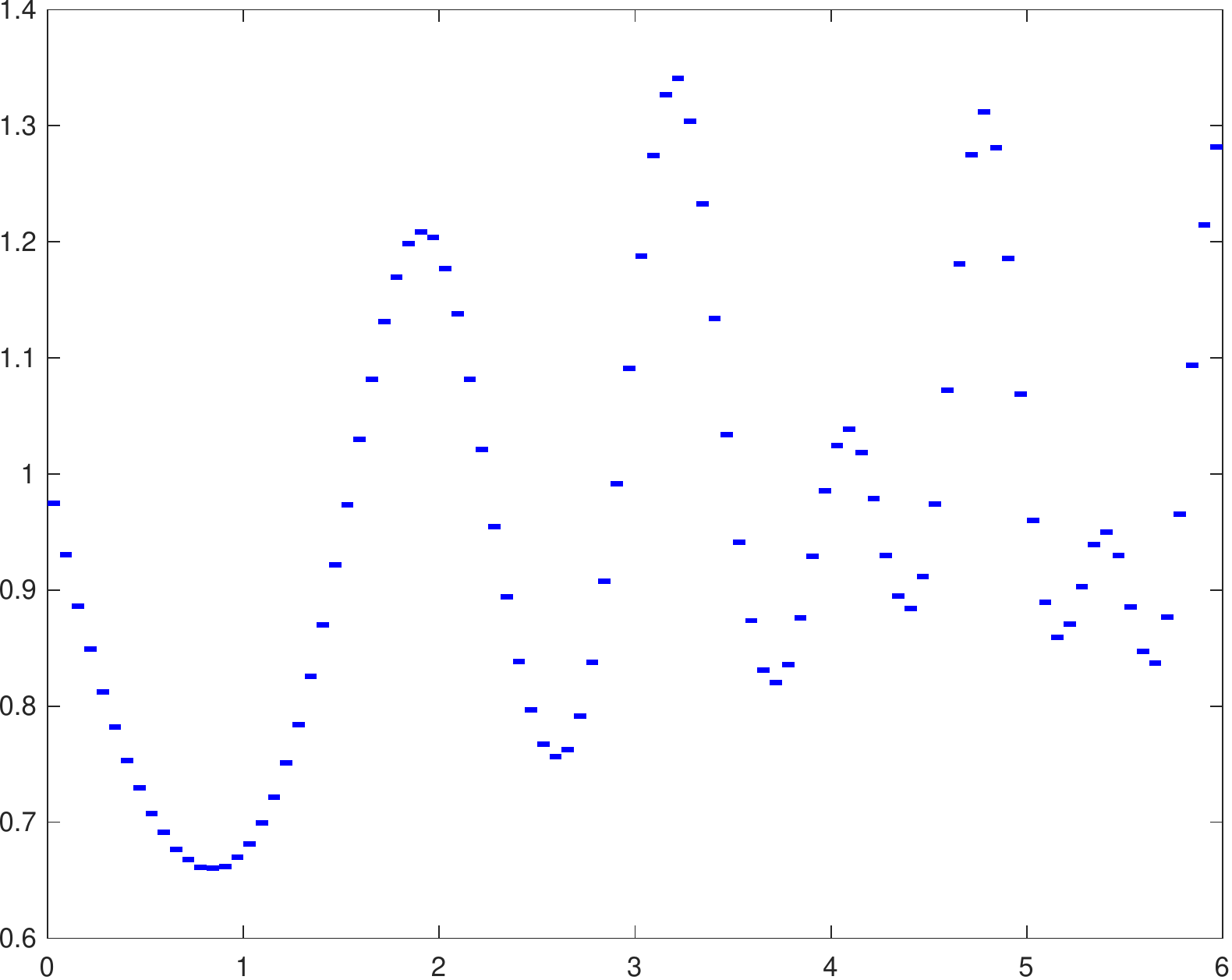}
\caption{Upper left: $h_{11}^{\pi}$. Upper right: $h_{11}^{2\pi}$. Lower left: $h_{11}^{4\pi}$. Lower right: $h_{11}^{8\pi}$.}
\end{figure}
\end{example}

\subsection{Beyond \texorpdfstring{$PW-$}{PW}measures}

The results of this section show that in some cases the periodization approach works beyond the $PW-$class. 

It follows from theorem \ref{PWcondition} that any measure decaying to $0$ at infinity is not a PW-measure. Nonetheless, we have the following statement.

\begin{theorem} \label{decay}
Let $\mu$ be a positive even Poisson-finite measure on $\R$. Suppose that there exist $C,c > 0$, such that $\mu$ has infinite support on $[-C, C]$, and $\mu(x, x + c)$ is decreasing on $\mathbb{R}_+$. Then
$$h^{nc}_{11} \overset{\ast}{\to} h_{11}\text{ as }n\to\infty,\ n\in\mathbb N.$$
\end{theorem}

\begin{proof}
Note that when $nc > C$, the periodization $\mu_{nc}$ is a PW-measure. Then we can use the algorithm described in Section \ref{PeriodicPW} to recover $h_{11}^{nc}$.

By lemma \ref{SuffCondition}, it suffices to show that $\lim_{n \to \infty} K_0^{t, nc}(0) = K_0^t(0)$ for every $t$. We will use the following equivalent definition of $K_0^t(0)$ (see for instance \cite{MP}): \begin{equation*}
    K_0^t(0) = \sup\{|f(0)|\ |\ \norm{f}_{L^2(\mu)} \leqslant 1, f \in PW_t\},
\end{equation*}
and 
\begin{equation*}
    K_0^{t, nc}(0) = \sup\{|f(0)|\ |\ \norm{f}_{L^2(\mu_{nc})} \leqslant 1, f \in PW_t\}.
\end{equation*}

Suppose that $$\lim_{n \to \infty} K_0^{t, nc}(0) \neq K_0^t(0).$$ 
Then there exist $\epsilon > 0$, and $\{ n_k \}_{k \in \mathbb{N}}$ such that $$\lim_{k \to \infty} n_k c = \infty$$ and \begin{equation}\label{decayfail}
    \abs{K_0^{t, n_k c}(0) - K_0^t(0)} > \epsilon, \quad \text{for all } k.
\end{equation}

When $nc > C$, for every $f \in PW_t$, we have \begin{equation*}
\begin{split}
    \norm{f}_{L^2(\mu_{nc})}^2 &= \int_{-\infty}^\infty | f(x)|^2 d\mu_{nc}(x) = \int_{\abs{x} \geqslant nc} |f(x)|^2 d\mu_{nc}(x) + \int_{\abs{x} < nc} |f(x)|^2 d\mu_{nc}(x)\\
    &= \int_{\abs{x} \geqslant nc} |f(x)|^2 d\mu_{nc}(x) + \int_{\abs{x} < nc} |f(x)|^2 d\mu(x).
\end{split}
\end{equation*}

Notice that it follows from our assumption
that for any $S\subseteq \R_{\pm}$, $\mu(S\pm nc) \leqslant \mu(S)$ for every $n\in\mathbb N$. In particular, if $S \subseteq [nc, \infty)$,  or $S \subseteq (-\infty, -nc]$, then $\mu(S) \leqslant \mu_{nc}(S)$. Therefore, 
$$\int_{|x|\geqslant nc}|f(x)|^2d\mu_{nc}(x)\geqslant
\int_{|x|\geqslant nc}|f(x)|^2d\mu(x)$$
and
\begin{equation*}
     \norm{f}_{L^2(\mu_{nc})}^2 \geqslant \int_{\abs{x} \geqslant nc} \abs{f(x)}^2 d\mu(x) + \int_{\abs{x} < nc} \abs{f(x)}^2 d\mu(x) = \norm{f}_{L^2(\mu)},
\end{equation*}
which implies that $K_0^{t, nc}(0) \leqslant K_0^t(0)$.
It also follows that 
$$\int_{\abs{x} \geqslant mc} \abs{f(x)}^2 d\mu_{nc}(x)\leqslant \int_{\abs{x} \geqslant mc} \abs{f(x)}^2 d\mu_c(x)$$
for any $m,n\in\mathbb N$.

Since $K_0^{t, nc}(0) \leqslant K_0^t(0)$, we can rewrite \eqref{decayfail} as \begin{equation*}
    K_0^t(0) - K_0^{t, n_k c}(0) > \epsilon, \quad \text{for all } n.
\end{equation*}

Let $f \in PW_t$ such that $f(0) = K_0^t(0) - \epsilon/2$, and $\norm{f}_{L^2(\mu)} = 1$. Define 
$$f^{n_k c}(x) = f(x)/\norm{f}_{L^2(\mu_{n_k c})}$$
for all $k$. Since $f^{n_k c} \in PW_t$, and $\norm{f^{n_k c}}_{L^2(\mu_{n_k c})}^2 = 1$, we have $K_0^{t, n_k c}(0) \geqslant f^{n_k c}(0)$ for all $n$.  Also, \begin{equation*}
    \norm{f}_{L^2(\mu_{n_k c})}^2 - \norm{f}_{L^2(\mu)}^2 = \int_{\abs{x} > n_k c} \abs{f(x)}^2 d\mu_{n_k c}(x) - \int_{\abs{x} > n_k c} \abs{f(x)}^2 d\mu(x) \leqslant 2 \int_{\abs{x} > n_k c} \abs{f(x)}^2 d\mu_c(x).
\end{equation*}
Since $f\in PW_t$, $f'\in L^2(\R)$. It follows that as $k \to \infty$, 
$$\int_{\abs{x} > n_k c} \abs{f(x)}^2 d\mu_c(x) \to 0$$
via the standard argument used in the proof of theorem \ref{PW}.
This implies 
$$\norm{f}_{L^2(\mu_{n_k c})} \to \norm{f}_{L^2(\mu)} = 1.$$
Therefore, we have 
$$\lim_{k \to \infty} f^{n_k c}(0) = \lim_{k \to \infty} f(0)/\norm{f}_{L^2(\mu_{n_k c})} = f(0) = K_0^t(0) - \epsilon/2.$$
In particular, for $k$ big enough, we have \begin{equation*}
 \lim_{k \to \infty}   K_0^{t, n_k c}(0) \geqslant \lim_{k \to \infty} f^{n_kc}(0) \geqslant K_0^t(0) - \frac{3}{4} \epsilon.
\end{equation*}
This contradicts \begin{equation*}
    K_0^t(0) - K_0^{t, n_k c}(0) > \epsilon, \quad \text{for all } k.
\end{equation*}
Therefore, $$\lim_{n \to \infty} K_0^{t, nc}(0) = K_0^t(0).$$
\end{proof}

The following particular case can be proved in a similar way or deduced from the last statement.

\begin{corollary}\label{CorDecay}
Let $\mu$ be a positive even absolutely continuous locally-finite measure, whose density 
decreases on $\R_+$. Then
$$h_{11}^T\overset{\ast}{\to} h_{11}$$
as $T\to\infty$.
\end{corollary}

\begin{example}
Let us revisit example \ref{oneplusdelta}, where $h_{11}(t)$ was calculated explicitly for the spectral measure 
$$\mu = \sqrt{2\pi} \delta_0 + \frac{1}{\sqrt{2\pi}}m.$$
This measure satisfies the assumptions of both theorems \ref{PW} and \ref{decay}. We periodize the measure with $T = \pi, 2\pi, 4\pi, 8\pi$, and compare the $h_{11}^T$'s with $h_{11}(t) = \frac{\sqrt{2\pi}}{(1 + 2t)^2}$. In the following pictures, the black step functions are the $h_{11}^T$'s obtained from the periodizations, and the orange curve in every figure is the actual $h_{11}(t)$.

\begin{figure}[H]
\centering
\includegraphics[width=0.45\textwidth]{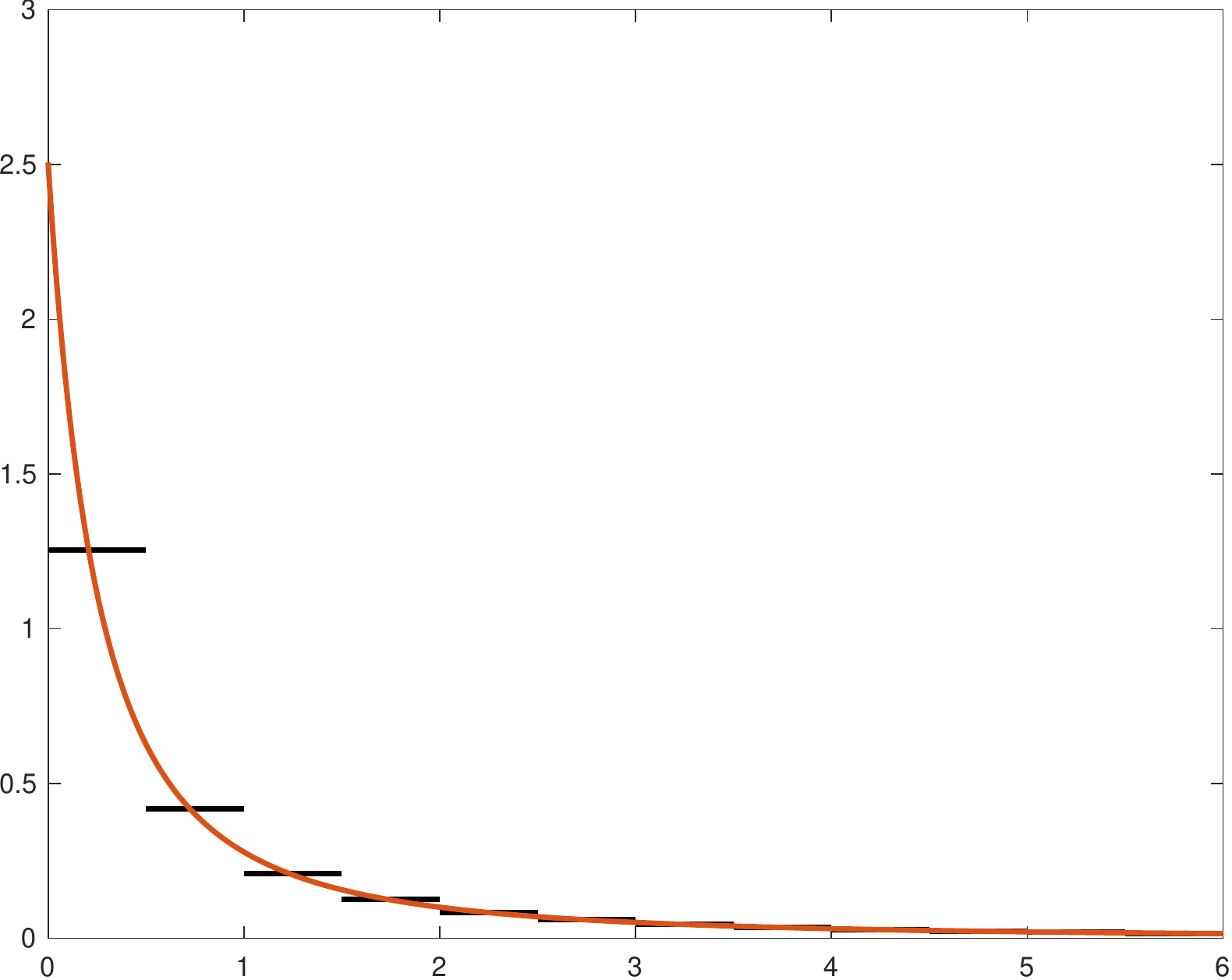}
\includegraphics[width=0.45\textwidth]{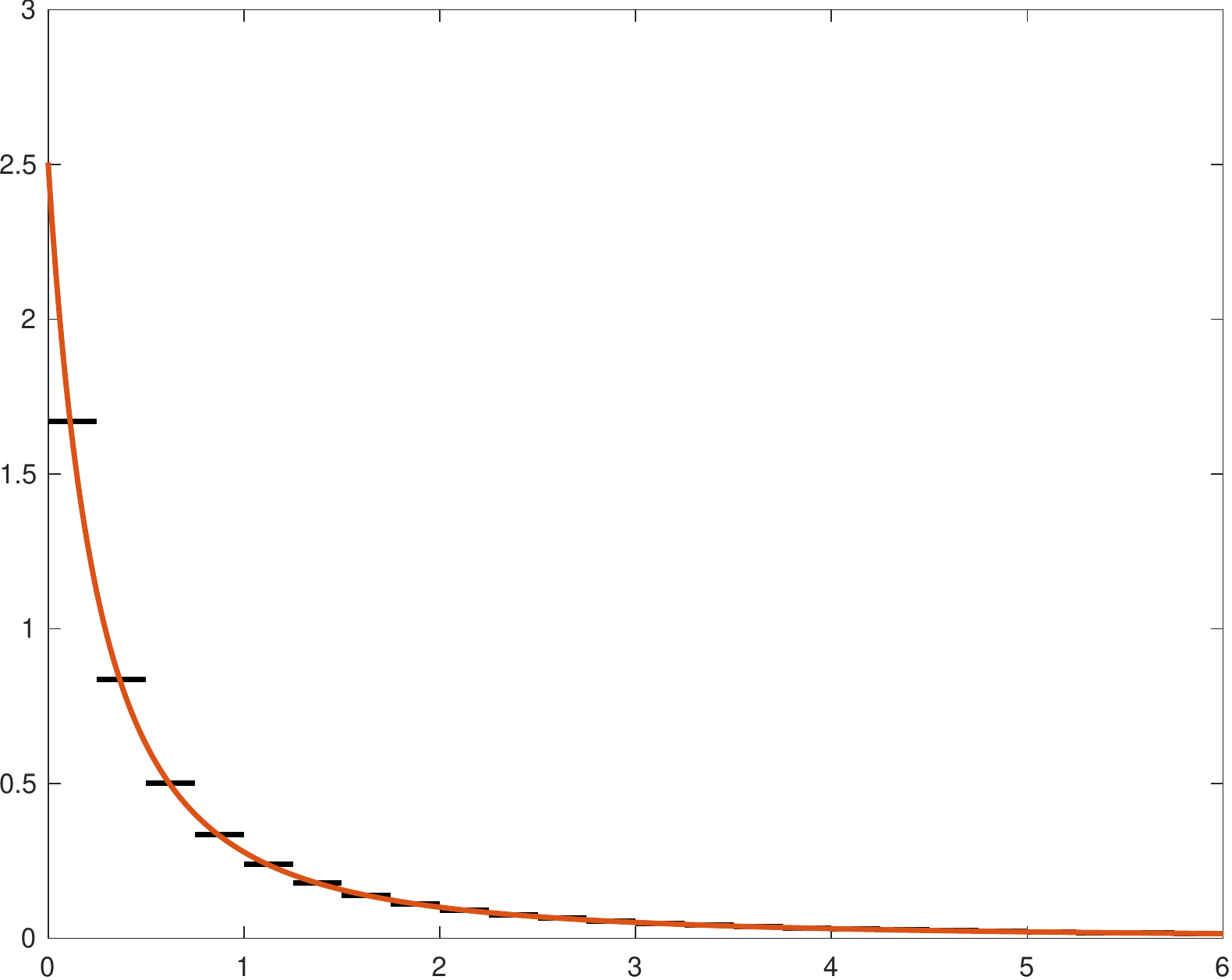}
\end{figure}

\begin{figure}[H]
\centering
\includegraphics[width=0.45\textwidth]{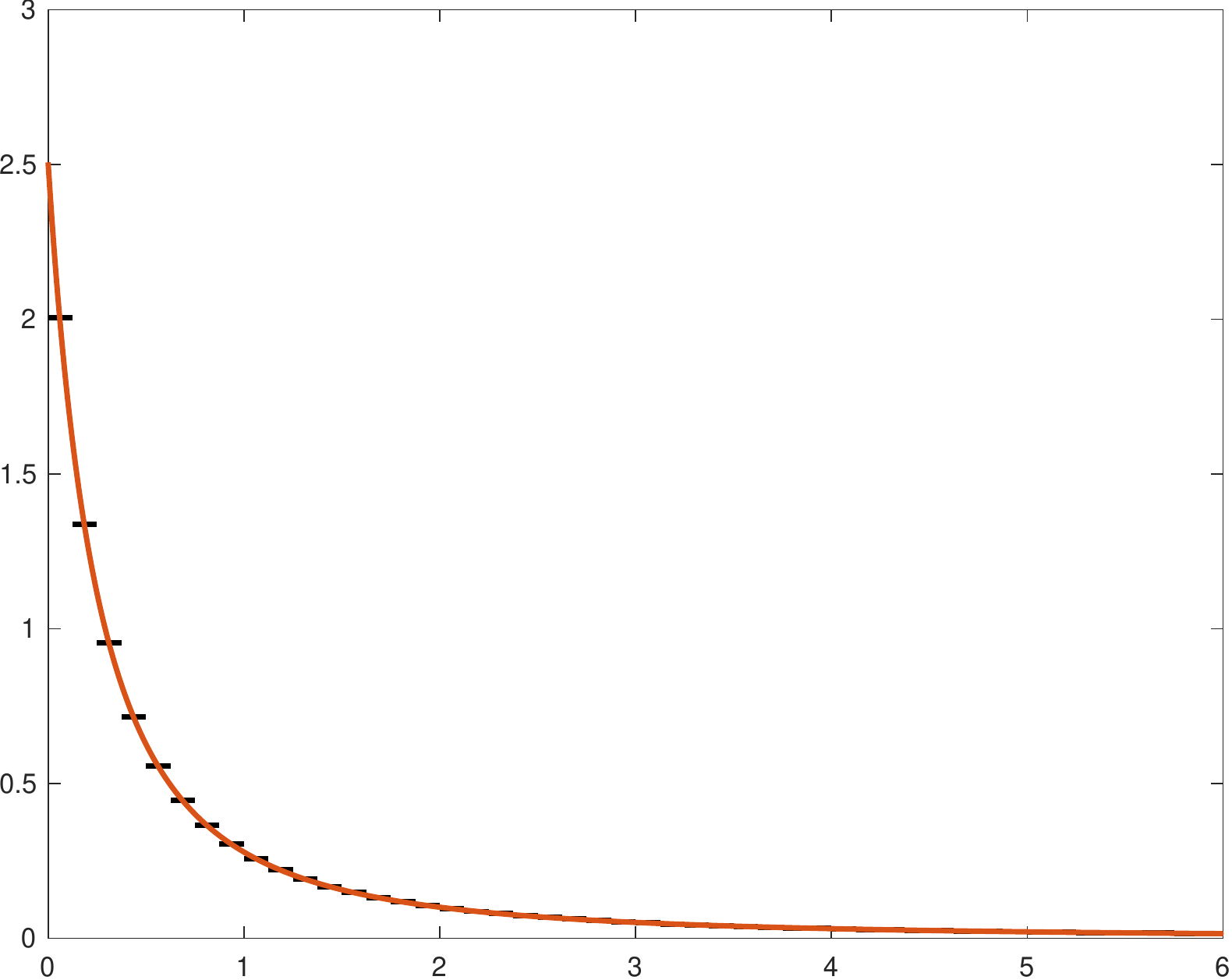}
\includegraphics[width=0.45\textwidth]{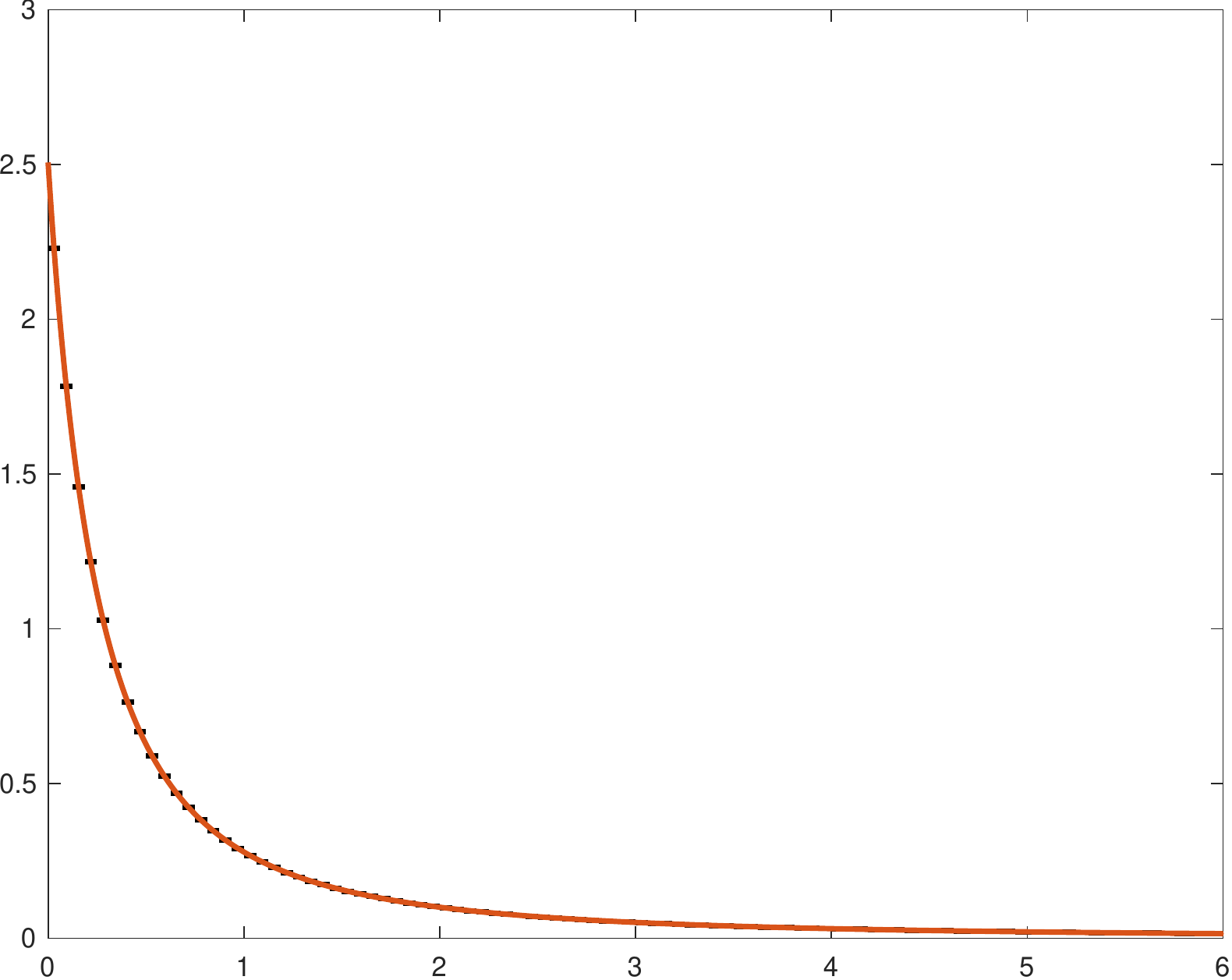}
\caption{Upper left: $h_{11}^{\pi}$. Upper right: $h_{11}^{2\pi}$. Lower left: $h_{11}^{4\pi}$. Lower right: $h_{11}^{8\pi}$.}
\end{figure}

Let $\mu_T$ be the $2T-$periodization of $\mu$, i.e., $\mu_T = \sqrt{2\pi}\delta_0 + \frac{1}{\sqrt{2\pi}}m$ on $[-T, T]$, and $\mu_T$ is a $2T-$periodic measure. Then the measure corresponding to $\mu_T$ on the unit circle $\mathbb{T}$ is $\mu_T^\mathbb{T} = \frac{\pi + T}{\sqrt{2\pi}} \left( \frac{\pi}{\pi + T} \delta_0 + \frac{T}{\pi + T} \frac{m}{2\pi} \right)$, where $\delta_0$ here stands for the unit point mass at $\theta = 0$ on $\mathbb{T}$. We know the orthonormal polynomials corresponding to this measure, for instance from \cite{Simon}. Therefore by theorem \ref{onpoly}, the $n$-th step of $h_{11}^T$ corresponding to $\mu_T$ takes value \begin{equation*}
    \frac{\sqrt{2\pi}T^2}{(n\pi + T)(n\pi + T + \pi)}.
\end{equation*}

For a given $T$, in the places of $n$, plug in $\frac{2Tt}{\pi} + 1$. This produces the function whose graph passes through all the right endpoints of the steps for $h_{11}^T$. Denoting this function by $\psi_T(t)$, we get 
\begin{equation*}
    \psi_T(t) = \frac{\sqrt{2\pi}T^2}{4t^2T^2 + 4tT^2 + 6\pi t T + T^2 + 3\pi T + 2\pi^2}.
\end{equation*}
After taking the limit as $T \to \infty$, we obtain the expected answer: 
$$\lim_{T \to \infty} \psi_T(t) = \frac{\sqrt{2\pi}}{4t^2 + 4t + 1} = \frac{\sqrt{2\pi}}{(2t + 1)^2} = h_{11}(t).$$

\end{example}

Our next theorem allows one to increase the class of measures for which the periodization approach works.

\begin{theorem}\label{polygrowth}
Let $\mu$ be an even Poisson-finite measure such that $d\mu = h(x) d\nu$, where $\nu$ satisfies the conditions of theorem \ref{decay}, and $h$ is an even positive function increasing on $\mathbb{R}_+$, and satisfies $h(x) = O(\abs{x}^q)$ as $|x|\to\infty$ for some $q>0$. Then $h_{11}^{nc} \overset{\ast}{\to} h_{11}$ as $n \to \infty$.
\end{theorem}

\begin{proof}
Without loss of generality, we may assume that $h(0) = 1$. If $h(0) \neq 1$, define $\Tilde{h}(x) = h(x)/h(0)$ for all $x$, and $\Tilde{\nu}(S) = h(0) \nu(S)$ for every $S \subseteq \mathbb{R}$.

Let $\mu_{nc}(x)$ be the periodization of $\mu$, and let $\nu_{nc}$ be the periodization of $\nu$. Then we have $ \mu_{nc} = h(x) \nu_{nc}$ on $[-nc, nc]$. Recall that $\nu_{nc}$ is a $2nc-$periodic PW-measure for $n$ big enough. Since for every set $S \subseteq [-nc, nc]$, $\nu_{nc}(S) \leqslant \mu_{nc}(S) \leqslant h(nc) \nu_{nc}(S)$, $\mu_{nc}$ is also a $2nc-$periodic PW-measure for $n$ big enough. Therefore, we can use the periodic algorithm described in Section \ref{PeriodicPW} to recover $h_{11}^{nc}$.

By lemma \ref{SuffCondition}, it suffices to show that $\lim_{n \to \infty} K_0^{t, nc}(0) = K_0^t(0)$ for almost every $t$. We will again use the following equivalent definition of $K_0^t(0)$: \begin{equation*}
    K_0^t(0) = \sup\{|f(0)|\ |\ \norm{f}_{L^2(\mu)} \leqslant 1, f \in PW_t\},
\end{equation*}
and 
\begin{equation*}
    K_0^{t, nc}(0) = \sup\{|f(0)|\ |\ \norm{f}_{L^2(\mu_{nc})} \leqslant 1, f \in PW_t\}.
\end{equation*}

Suppose $\lim_{n \to \infty} K_0^{t, nc}(0) \neq K_0^t(0)$. Then there exists $\epsilon > 0$, and $\{ n_k \}_{k \in \mathbb{N}}$ with $\lim_{k \to \infty} n_k = \infty$ such that \begin{equation}\label{limitfail}
    \abs{K_0^{t, n_kc}(0) - K_0^t(0)} > \epsilon, \quad \text{for all } k.
\end{equation}

Note that $K_0^s(0)$ is  an increasing function of $s$. Therefore it is continuous except possibly at a countable set of points. We assume that $K_0^s(0)$ and all $K_0^{s,n_k c}(0)$ are continuous at $t$.

Without loss of gererality, we consider the following two cases: \begin{enumerate}
    \item $K_0^{t, n_k c}(0) - K_0^t(0) > \epsilon$ for all $k$,
    \item $K_0^t(0) - K_0^{t, n_k c}(0) > \epsilon$ for all $k$.
\end{enumerate}

We start by showing that the first case can never happen. Let $f^{n_kc} \in PW_t$ with $f^{n_kc}(0) = K_0^t(0) + \epsilon/2$, and $\norm{f^{n_kc}}_{L^2(\mu_{n_kc})} \leqslant 1$. We identify the function $\frac{\sin(\delta x)}{\delta x}$ with its contionuous extension \begin{equation*}
    \begin{cases}
    \frac{\sin(\delta x)}{\delta x}, \quad &x \neq 0\\
    1 & x = 0
    \end{cases}.
\end{equation*} Fix $m > q$ and let $\delta > 0$. Then there exists $M_\delta$ such that for every $g \in PW_t$ and $M > M_\delta$, \begin{equation*}
    \int_{\abs{x} > M} \abs{g(x)}^2 \left( \frac{\sin(\delta x)}{\delta x} \right)^{2m} h(x) d\nu(x) \leqslant  \int_{\abs{x} > M} \abs{g(x)}^2 d\nu(x).
\end{equation*}
Choose $k$ large enough such that $n_k c > M_{\delta}$. Then \begin{equation*}
\begin{split}
    &\norm{f^{n_k c}(x) \left( \frac{\sin(\delta x)}{\delta x} \right)^m}_{L^2(\mu)}^2 = \int_{-\infty}^\infty \abs{ f^{n_k c}(x) }^2 \left( \frac{\sin(\delta x)}{\delta x} \right)^{2m} h(x) d\nu(x)\\
    &= \int_{\abs{x} > n_k c}  \abs{f^{n_k c}(x) }^2 \left( \frac{\sin(\delta x)}{\delta x} \right)^{2m} h(x) d\nu(x) + \int_{\abs{x} \leqslant n_k c}  \abs{ f^{n_k c}(x) }^2 \left( \frac{\sin(\delta x)}{\delta x} \right)^{2m} h(x) d\nu(x) \\
    & \leqslant \int_{\abs{x} > n_k c} \abs{ f^{n_k c}(x) }^2 d\nu(x) + \int_{\abs{x} \leqslant n_k c}  \abs{ f^{n_k c}(x) }^2 h(x) d\nu(x).
\end{split}    
\end{equation*}

Recalling that $\nu$ is satisfies the conditions of theorem \ref{decay}, similarly to its proof we obtain \begin{equation*}
    \int_{\abs{x} > n_k c} \abs{ f^{n_k c}(x) }^2 d\nu(x) \leqslant \int_{\abs{x} > n_k c} \abs{ f^{n_k c}(x) }^2 d\nu_{n_k c}(x).
\end{equation*}

Also recall that $d\mu(x) = h(x) d\nu(x)$, so \begin{equation*}
    \int_{\abs{x} \leqslant n_k c}  \abs{ f^{n_k c}(x) }^2 h(x) d\nu(x) = \int_{\abs{x} \leqslant n_k c}  \abs{ f^{n_k c}(x) }^2 d\mu(x) = \int_{\abs{x} \leqslant n_k c}  \abs{ f^{n_k c}(x) }^2 d\mu_{n_k c}(x).
\end{equation*}

Therefore, \begin{equation*}
\begin{split}
    &\norm{f^{n_k c}(x) \left( \frac{\sin(\delta x)}{\delta x} \right)^m}_{L^2(\mu)}^2 \leqslant \int_{\abs{x} > n_k c} \abs{ f^{n_k c}(x) }^2 d\nu_{n_k c}(x) + \int_{\abs{x} \leqslant n_k c}  \abs{ f^{n_k c}(x) }^2 d\mu_{n_k c}(x) \\
    &\leqslant \int_{\abs{x} > n_k c} \abs{ f^{n_k c}(x) }^2 d\mu_{n_k c}(x) + \int_{\abs{x} \leqslant n_k c}  \abs{ f^{n_k c}(x) }^2 d\mu_{n_k c}(x) = \int_\mathbb{R}  \abs{ f^{n_k c}(x) }^2 d\mu_{n_k c}(x) \leqslant 1.
\end{split}    
\end{equation*}

Since $f^{n_k c}(x) \left( \frac{\sin(\delta x)}{\delta x} \right)^m \in PW_{t + \delta m}$, we have for every $\delta > 0$, \begin{equation*}
    K_0^{t + \delta m}(0) \geqslant f^{n_k c}(0) = K_0^t(0) + \epsilon/2.
\end{equation*}
By our assumption $s=t$ is a point of continuity of $K_0^s(0)$ as a function of $s$. Therefore, letting $\delta$ tend to $0$, we obtain
$$K_0^t(0) \geqslant K_0^t(0) + \epsilon/2$$
for some $\epsilon > 0$, which is a contradiction.

Now we move on to the second case. Let $f \in PW_t$ with $$f(0) = K_0^t(0) - \epsilon/2,\text{ and } \norm{f}_{L^2(\mu)} \leqslant 1.$$ 
We write the $2T-$periodization of $\mu$ as $d\mu_{T}(x) = h_{T}(x) d\nu_{T}(x)$. 

Like in the previous case, fix $m > q$ and let $\delta > 0$. There exists $M_\delta$ such that for every $g \in PW_t$ and $M > M_\delta$, \begin{equation*}
    \int_{\abs{x} > M} \abs{g(x)}^2 \left( \frac{\sin(\delta x)}{\delta x} \right)^{2m} h(x) d\nu(x) \leqslant  \int_{\abs{x} > M} \abs{g(x)}^2 d\nu(x).
\end{equation*}
It follows from our assumption on $h(x)$ that for every $k$, \begin{equation*}
    \int_{\abs{x} > M} \abs{g(x)}^2 \left( \frac{\sin(\delta x)}{\delta x} \right)^{2m} h_{n_k c}(x) d\nu_{n_k c}(x) \leqslant  \int_{\abs{x} > M} \abs{g(x)}^2 d\nu_{n_k c}(x).
\end{equation*}

Let $k$ be big enough so that $n_k c > M_\delta$. Then \begin{equation*}
\begin{split}
    &\abs{\norm{f(x) \left(\frac{\sin(\delta x)}{\delta x}\right)^m}_{L^2(\mu_{n_k c})}^2 - \norm{f(x) \left(\frac{\sin(\delta x)}{\delta x}\right)^m}_{L^2(\mu)}^2} \\
    &= \left|\int_{\abs{x} > n_k c} \abs{f(x)}^2 \left( \frac{\sin(\delta x)}{\delta x} \right)^{2m} h_{n_k c}(x) d\nu_{n_k c}(x) - \int_{\abs{x} > n_k c} \abs{f(x)}^2 \left( \frac{\sin(\delta x)}{\delta x} \right)^{2m} h(x) d\nu(x) \right| \\
    &\leqslant \int_{\abs{x} > n_k c} \abs{f(x)}^2 d\nu_{n_k c}(x) + \int_{\abs{x} > n_k c} \abs{f(x)}^2 d\nu(x) \leqslant 2 \int_{\abs{x} > n_k c} \abs{f(x)}^2 d\nu_c(x) \to 0 \text{~as~} k \to \infty.
\end{split}    
\end{equation*}

This implies \begin{equation*}
    \lim_{k \to \infty} \norm{f(x) \left(\frac{\sin(\delta x)}{\delta x}\right)^m}_{L^2(\mu_{n_k c})} =  \norm{f(x) \left(\frac{\sin(\delta x)}{\delta x}\right)^m}_{L^2(\mu)} < \norm{f(x)}_{L^2(\mu)} \leqslant 1.
\end{equation*}

Therefore for $k$ big enough, $K_0^{t + \delta m, n_k c}(0) \geqslant f(0) = K_0^t(0) - \epsilon/2$. Tending
$\delta$ and then $\epsilon$ to $0$ we obtain the statement.

\end{proof}

The following particular case can be proved in a similar way or deduced from the last statement.

\begin{corollary}\label{CorPolygrowth}
Let $\mu$ be a positive even absolutely continuous locally-finite measure, whose density is the product of $h(x)$, an even positive function increasing on $\mathbb{R}_+$, satisfying $h(x) = O(\abs{x}^q)$ as $|x|\to\infty$ for some $q>0$, and $f(x)$, an even non-negative function decreasing on $\mathbb{R}_+$. Then
$$h_{11}^T\overset{\ast}{\to} h_{11}$$
as $T\to\infty$.
\end{corollary}

We finish with the following numerical illustrations for the last theorem. 

\begin{example}
Consider the canonical system with $d\mu(x) = (1 + \abs{x}^\frac{1}{2})dx$. This measure is not a PW-measure, but satisfies the assumptions of corollary \ref{CorPolygrowth}.

Here are four approximations with $T = \pi, 2\pi, 4\pi, 8\pi$.

\begin{figure}[H]
\centering
\includegraphics[width=0.45\textwidth]{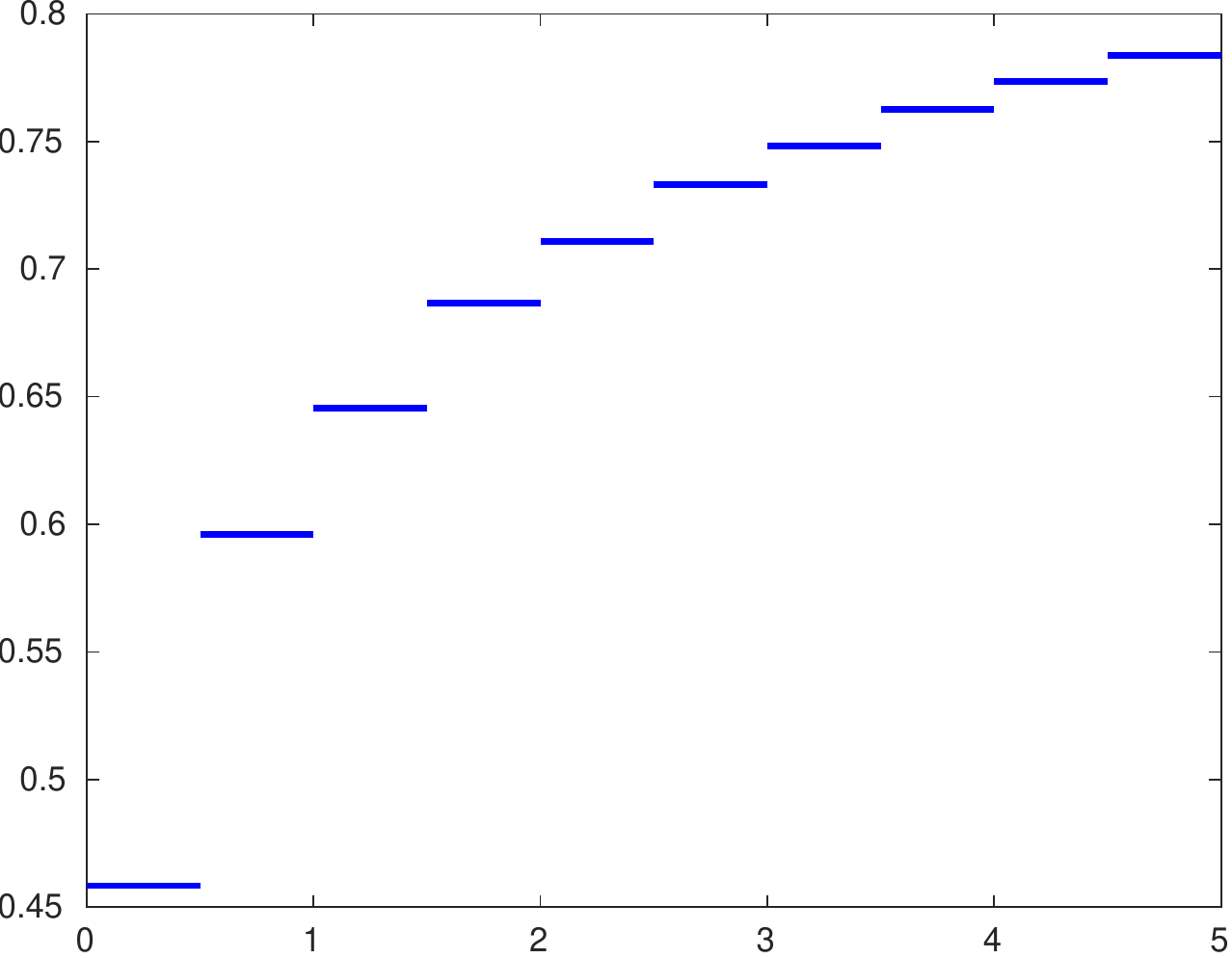}
\includegraphics[width=0.45\textwidth]{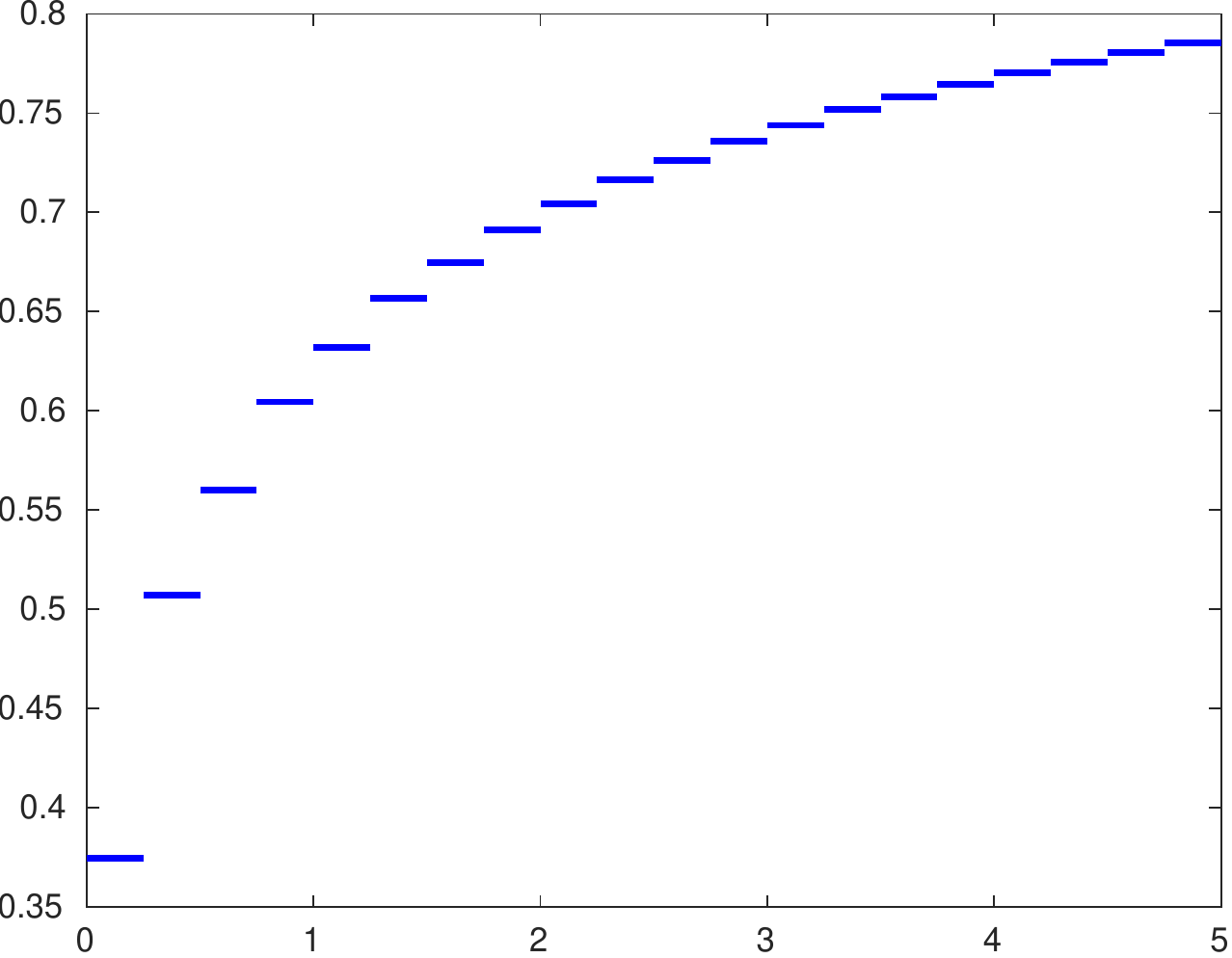}
\includegraphics[width=0.45\textwidth]{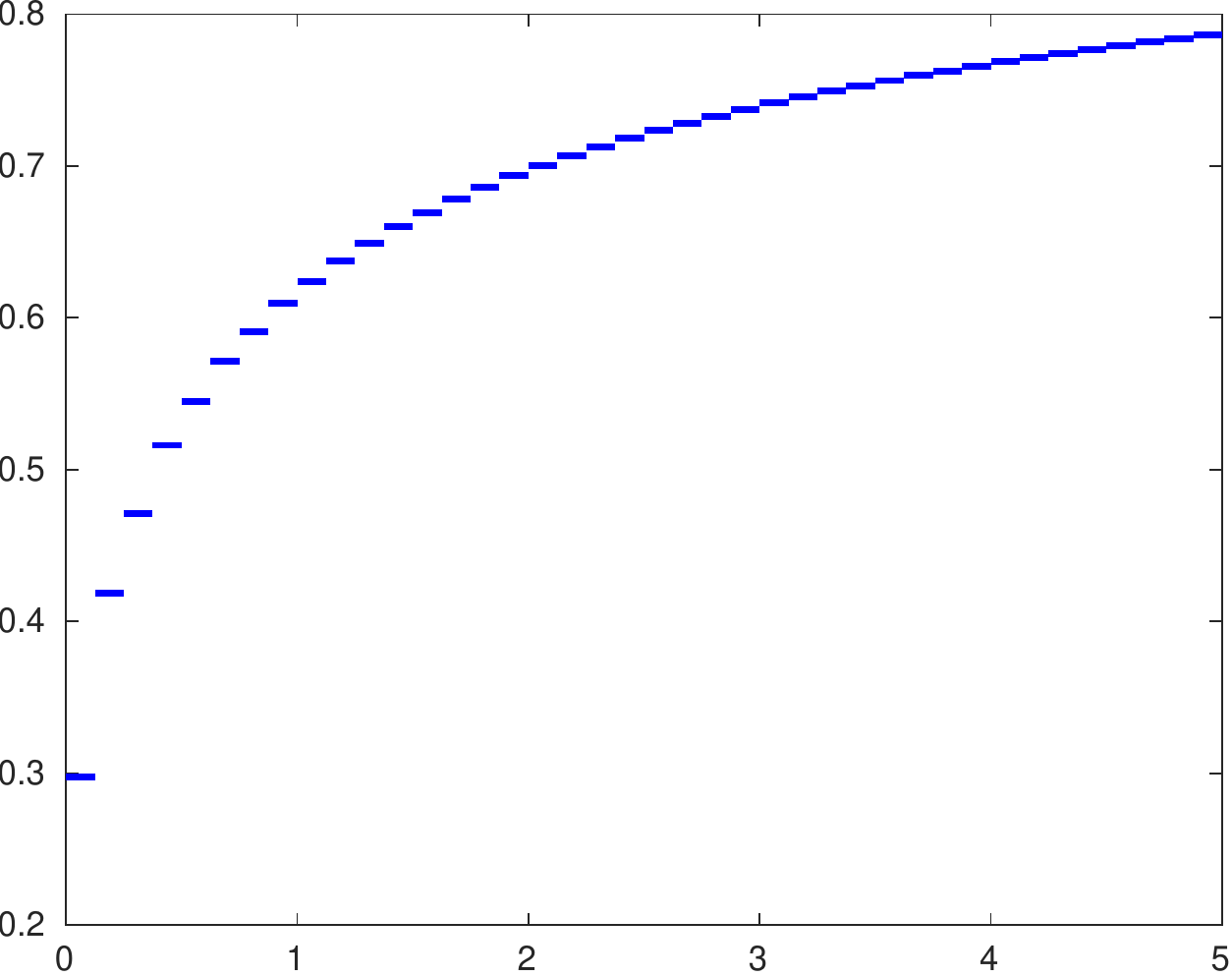}
\includegraphics[width=0.45\textwidth]{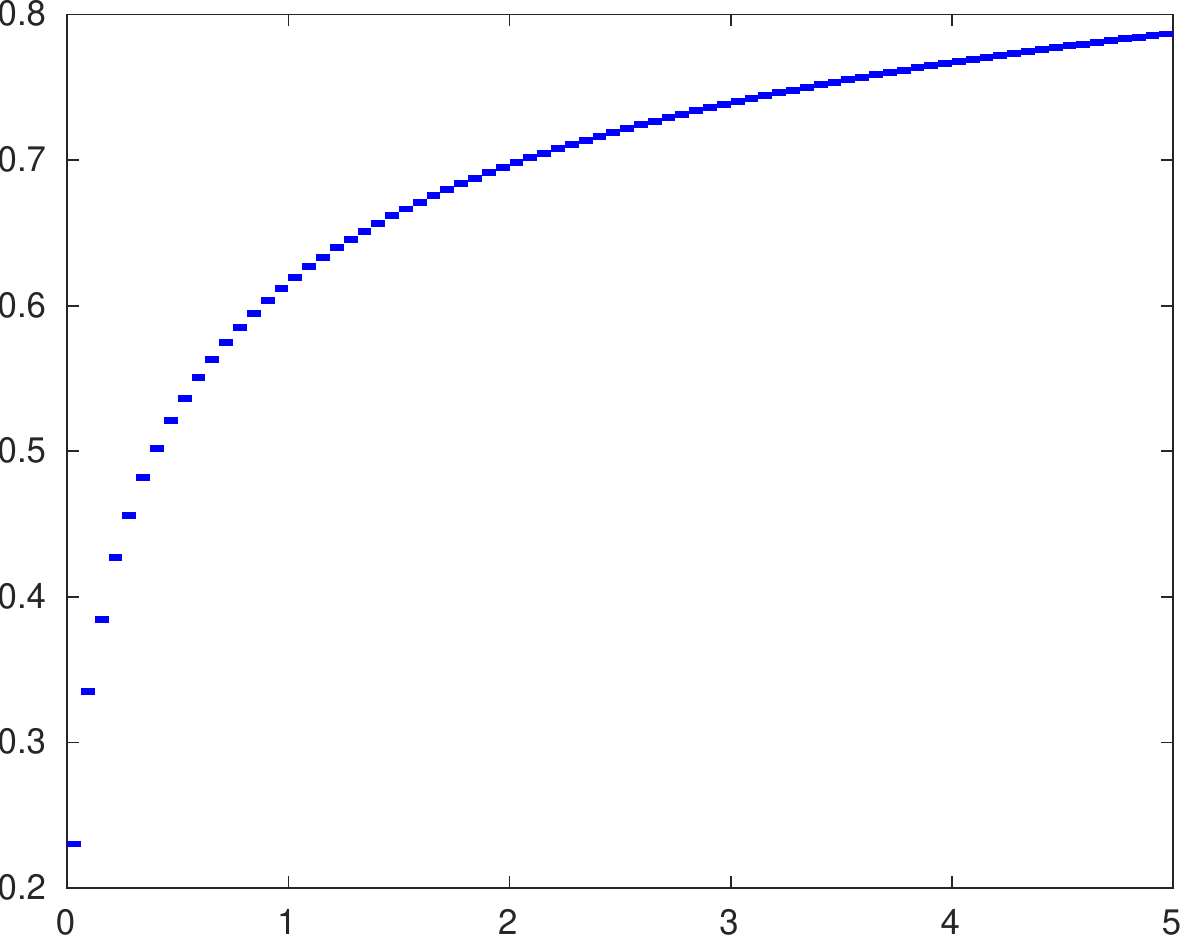}
\caption{Upper left: $h_{11}^{\pi}$. Upper right: $h_{11}^{2\pi}$. Lower left: $h_{11}^{4\pi}$. Lower right: $h_{11}^{8\pi}$.}
\end{figure}
\end{example}

\begin{example}
Consider the canonical system with $\mu(x) = (1 + \abs{x})^{\frac{1}{4}} m + \delta_0$, where $m$ is the Lebesgue measure and $\delta_0$ is the unit point mass at $0$. This measure is not a PW-measure, but satisfies the assumptions of theorem \ref{polygrowth}.

Here are four approximations with $T = \pi, 2\pi, 4\pi, 8\pi$.

\begin{figure}[H]
\centering
\includegraphics[width=0.45\textwidth]{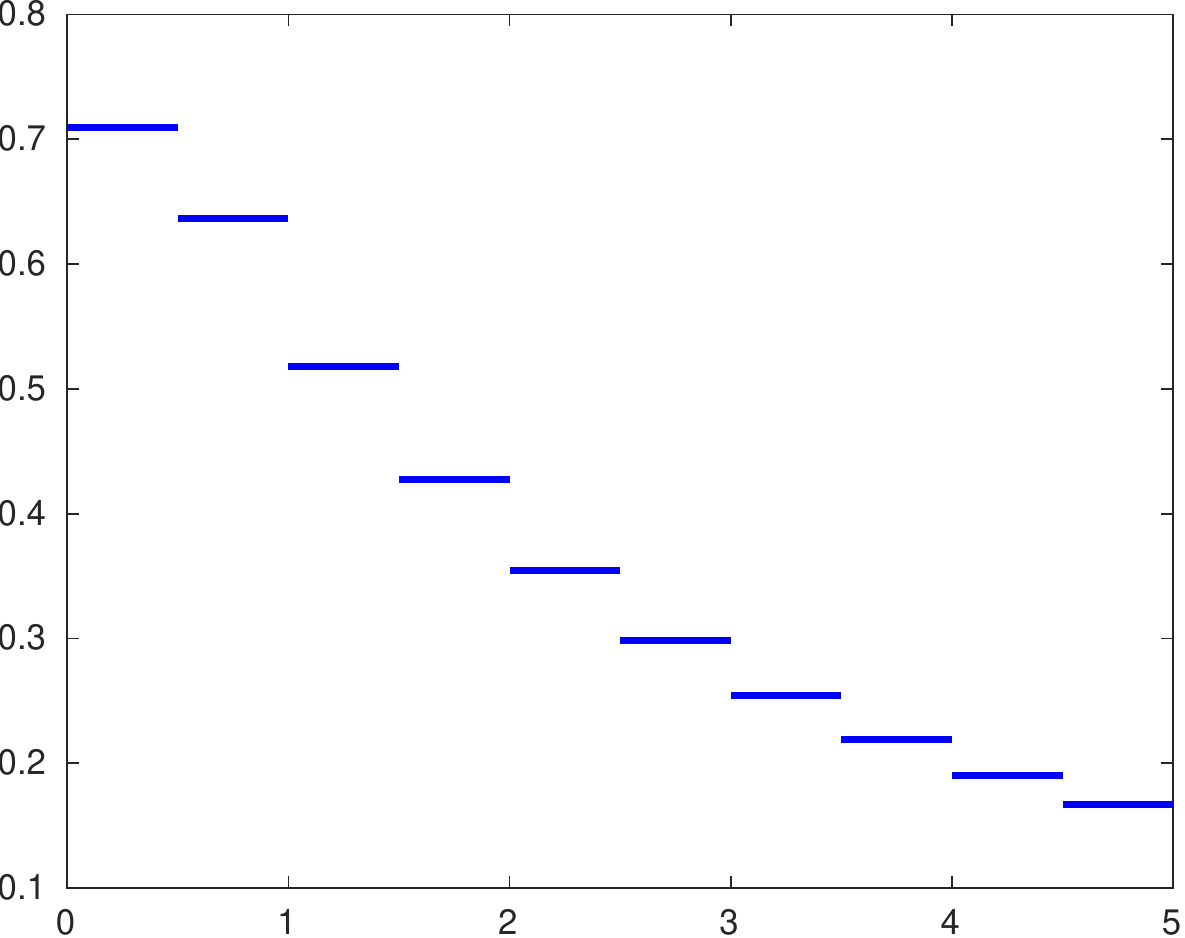}
\includegraphics[width=0.45\textwidth]{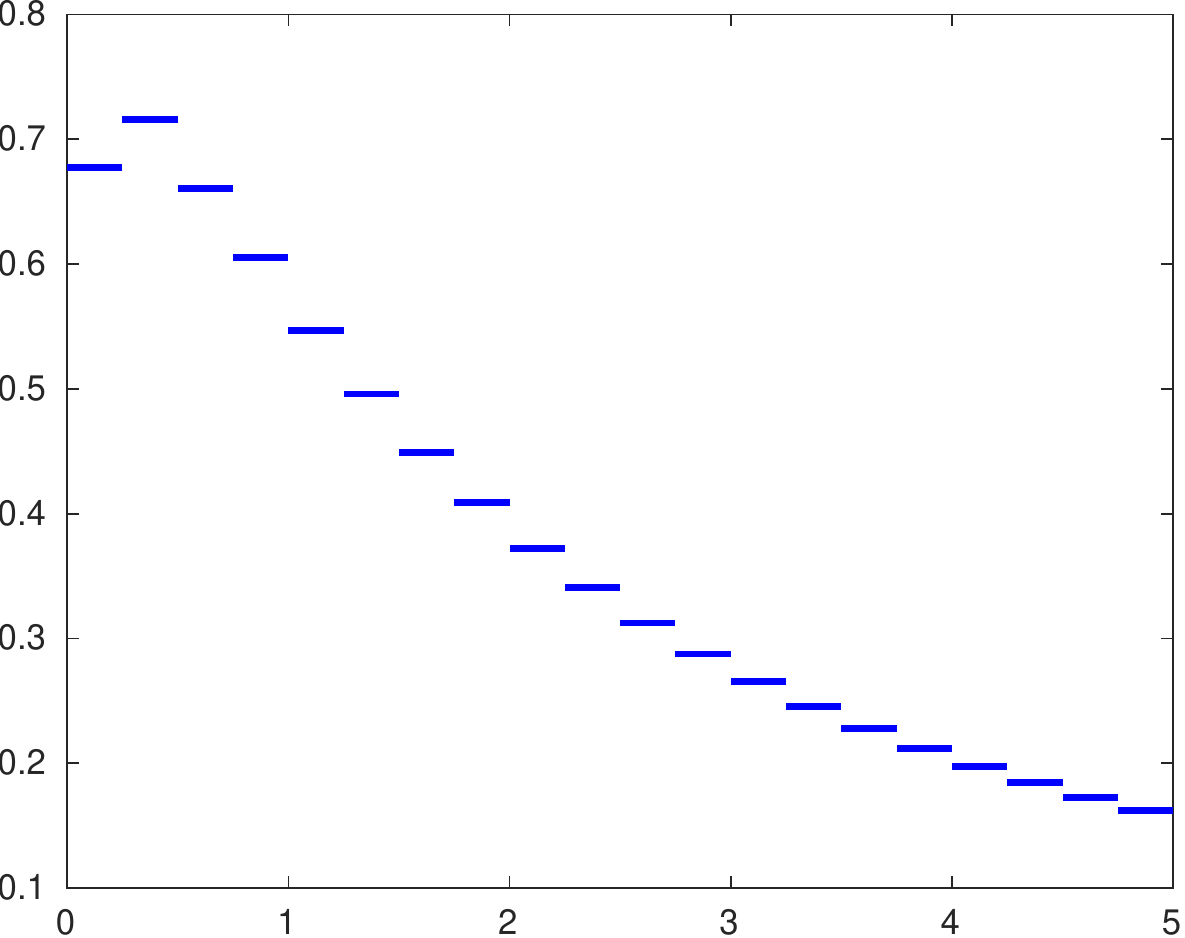}
\includegraphics[width=0.45\textwidth]{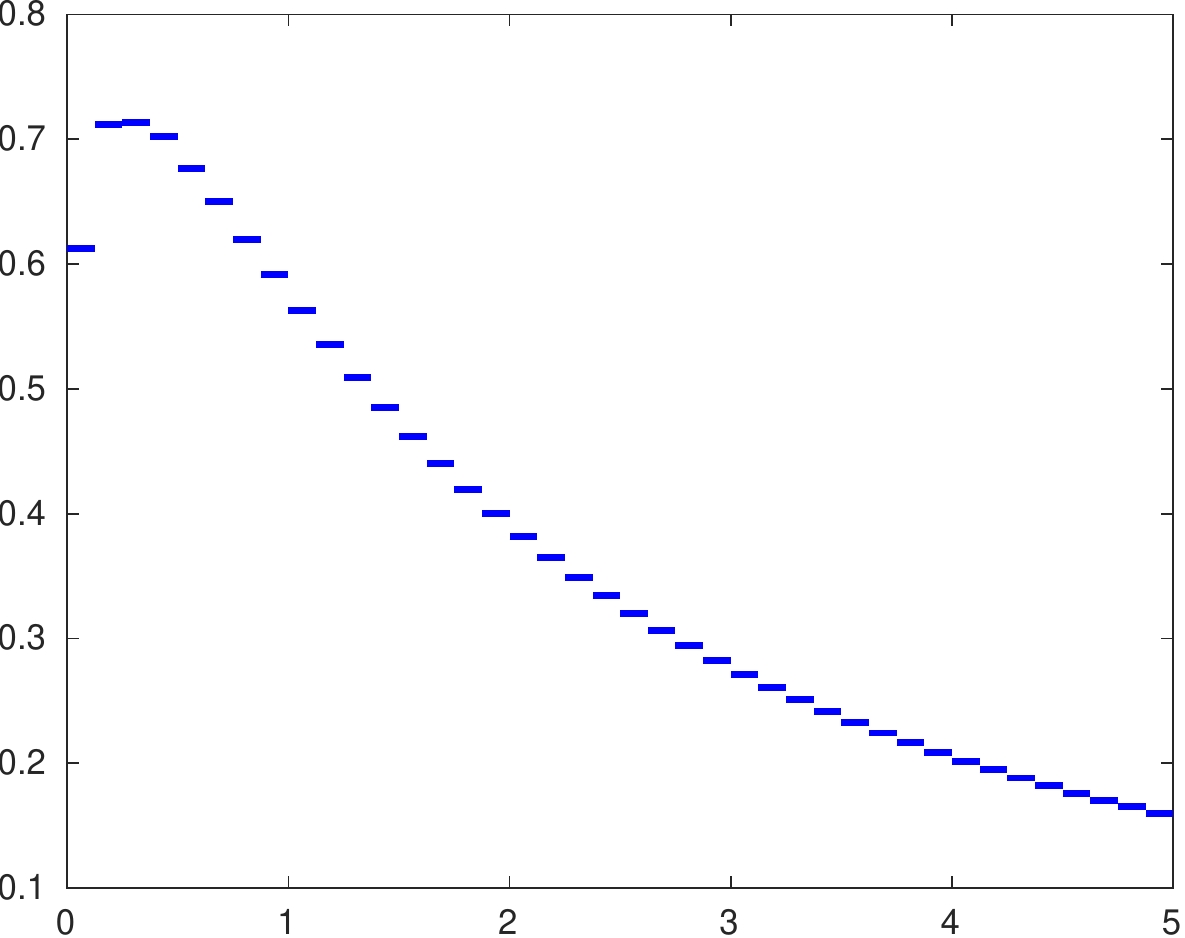}
\includegraphics[width=0.45\textwidth]{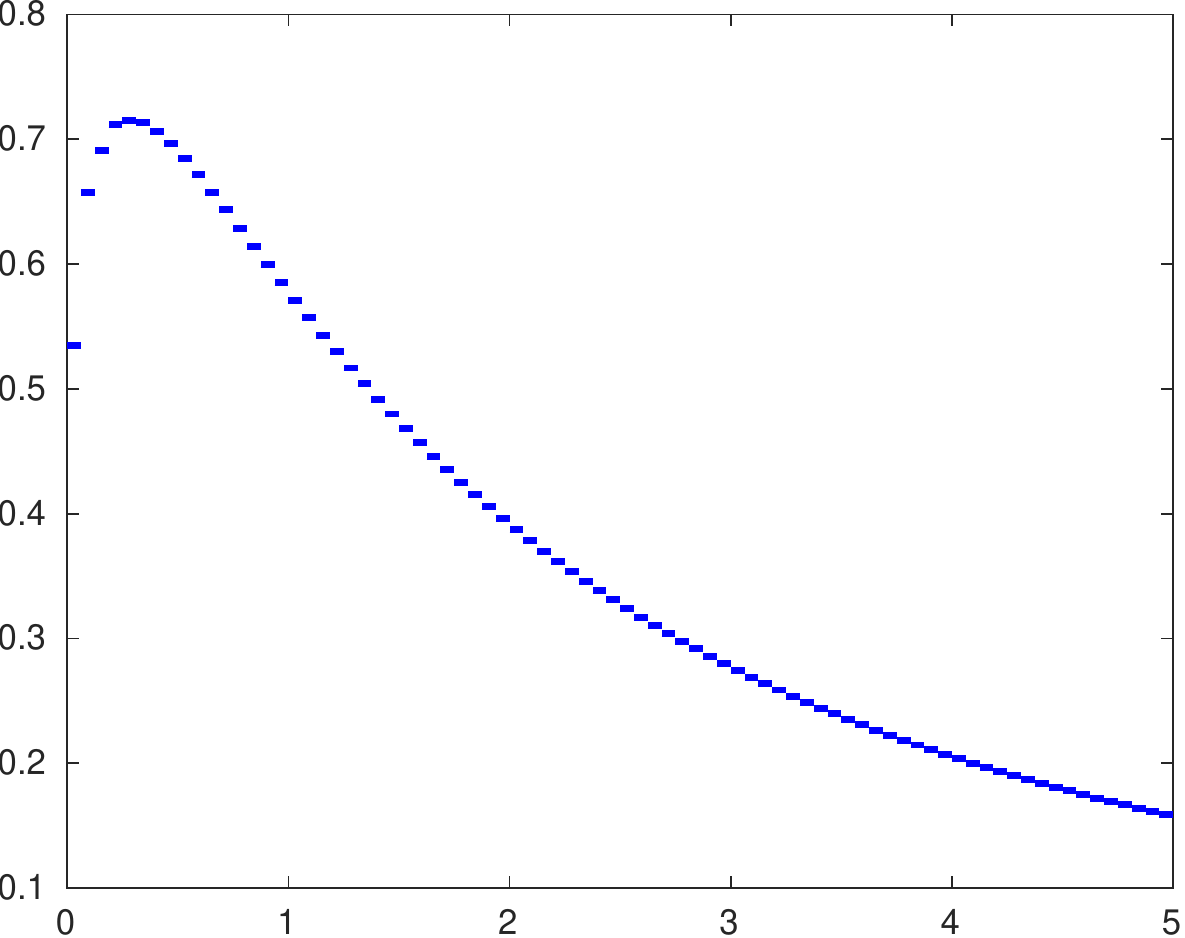}
\caption{Upper left: $h_{11}^{\pi}$. Upper right: $h_{11}^{2\pi}$. Lower left: $h_{11}^{4\pi}$. Lower right: $h_{11}^{8\pi}$.}
\end{figure}
\end{example}

\end{document}